\documentclass[twoside,11pt,abbrvbib]{article}

\usepackage{jmlr2e}
\RequirePackage[OT1]{fontenc}
\usepackage{amsmath,amsfonts,natbib}

\usepackage[utf8x]{inputenc}
\usepackage{amsmath}
\usepackage{amsfonts}
\usepackage{amssymb}
\usepackage{tikz}
\usepackage{tikz-cd}
\usepackage{textcomp}
\usepackage{pgfplots}
\usetikzlibrary{shapes.geometric}
\usetikzlibrary{arrows}
\usetikzlibrary{arrows.meta}

\usepackage{natbib}

\ShortHeadings{What can be estimated?}{Maclaren and Nicholson}
\firstpageno{1}

\begin{document}

\title{What can be estimated? Identifiability, estimability, causal inference and ill-posed inverse problems}

\author{\name Oliver J. Maclaren \email oliver.maclaren@auckland.ac.nz \\
       \addr Department of Engineering Science\\
       University of Auckland\\
       Auckland 1142, New Zealand.
       \AND
       \name Ruanui Nicholson \email ruanui.nicholson@auckland.ac.nz \\
       \addr  Department of Engineering Science\\
       University of Auckland\\
       Auckland 1142, New Zealand.}

\editor{TBD}

\maketitle


\begin{abstract}
    We consider basic conceptual questions concerning the relationship between statistical estimation and causal inference. Firstly, we show how to translate causal inference problems into an abstract statistical formalism without requiring any structure beyond an arbitrarily-indexed family of probability models. The formalism is simple but can incorporate a variety of causal modelling frameworks, including `structural causal models', but also models expressed in terms of, e.g., differential equations. We focus primarily on the structural/graphical causal modelling literature, however. Secondly, we consider the extent to which causal and statistical concerns can be cleanly separated, examining the fundamental question: `What can be estimated from data?'. We call this the problem of estimability. We approach this by analysing a standard formal definition of `can be estimated' commonly adopted in the causal inference literature -- identifiability -- in our abstract statistical formalism. We use elementary category theory to show that identifiability implies the existence of a Fisher-consistent estimator, but also show that this estimator may be discontinuous, and thus unstable, in general. This difficulty arises because the causal inference problem is, in general, an ill-posed inverse problem. Inverse problems have three conditions which must be satisfied to be considered well-posed: existence, uniqueness, and stability of solutions. Here identifiability corresponds to the question of uniqueness; in contrast, we take estimability to mean satisfaction of all three conditions, i.e. well-posedness. Lack of stability implies that naive translation of a causally identifiable quantity into an achievable statistical estimation target may prove impossible. Our article is primarily expository and aimed at unifying ideas from multiple fields, though we provide new constructions and proofs.
\end{abstract}

\maketitle



\section{Introduction}
A common idea in much of the causal inference literature \citep[see e.g.][and related work]{Pearl2009-qh} is that there is a natural separation of concerns between causal inference and statistical estimation of the form: 

\begin{itemize}
    \item First determine what can be estimated \textit{in principle} from the data using formal logics of identifiability analysis.
    \item Then call up the Statistics department to help design a \textit{particular} efficient (or otherwise desirable) estimator. 
\end{itemize}
Furthermore, it is often assumed that statistical and causal concepts should be kept separate because the standard formalisms of statistics cannot naturally express or accommodate causal information, unless additional concepts such as potential outcomes are included or new formalisms such as structural causal models are adopted \citep{Pearl2009-qh}.

Here we examine the extent to which this separation is required and/or desirable. Firstly, we show how a general abstract statistical formalism is capable of accommodating `causal queries' arising from a diversity of causal modelling formalisms, including so-called structural and graphical causal models. While we place particular emphasis on accommodating structural/graphical causal modelling, we also deliberately aim to not restrict ourselves to this particular causal modelling formalism: for example, this framework can also accommodate causal modelling approaches based on differential equations or individual-based models. In fact, our main inspiration for introducing this formalism comes from work on inverse problems where causal models from the physical sciences, typically expressed as differential equations, are combined with statistical estimation using noisy, indirect measurements of the quantities of interest \citep{Evans2002-as}. We then consider a key purported dividing line between statistical and causal questions in detail: identifiability. This is because in much of the causal inference literature, particularly that focused on formal logics of causal inference, identifiability of a quantity is taken to be synonymous with `estimable from data'. For example \citet[][p. 583]{Pearl2014-vp} give the following definition and description:

\begin{quote}
The following definition \textbf{captures the requirement that $Q$ be estimable from the data} [emphasis ours]:

\begin{definition}[Identifiability]
Causal query $Q(M)$ is identifiable, given a set of assumptions $A$, if for any two (fully-specified) models, $M_1$ and $M_2$, that satisfy $A$, we have

$$ P (M_1) = P (M_2) ⇒ Q(M_1) = Q(M_2).$$
\end{definition}
\end{quote}
(In the above, $P(M_1), P(M_2)$ denote the probability distributions \textit{implied by} the causal models $M_1, M_2$, rather than the probabilities \textit{of} the models.)

The implied equivalence between the notion `can be estimated from data' and the formalised concept of identifiability has led, it seems, to a strong focus in much of the formal causal literature on identification conditions for various quantities (or `causal queries'), and is at the heart of the idea that a `separation of concerns' between statistical and causal modelling is both possible and desirable. The culmination of this focus on identifiability in formal causal logics consists of proofs of the completeness of the do-calculus for answering identification questions \citep{shpitser2006identification,huang2008completeness}. Impressively, these identifiability results are \textit{entirely independent of the particular functional form of the causal model} (other than e.g. which variables depend on which) and thus can be fully summarised in minimalist `nonparametric' structural equations or graphical terms. The latter consist of directed acyclic graphs (DAGs).

In the present work we re-examine the fundamental mathematical ideas behind the assumed close connection between identifiability and `can be estimated from data'. We further abstract a statistical estimation framework used to express inverse problems as statistical estimation problems \citep{Evans2002-as}, and prove a number of basic results in a very general form using elementary category theory. In particular, we use this framework to formally prove results connecting identifiability to the existence of functional estimators. At the same time, we show that these estimators lack stability and/or statistical guarantees and in fact illustrate that identifiability is an inadequate characterisation of the more general concept of `can be estimated from the data' (even in principle), or what we will call \textit{estimability}. We hence give a more appropriate characterisation of estimability and illustrate the difference between this and identifiability with simple concrete examples. These results imply that, while identifiability conditions hold independently of the particular function classes to which causal models belong, estimability (even in principle) requires restrictions on this class. An alternative interpretation is that one \textit{only answer a much more restricted set of causal questions} than is usually supposed, where these restrictions are more stringent than those imposed by the requirement of identifiability. Thus we find that, in general, we cannot trust identifiability results to tell us what can be estimated, or which causal questions are answerable, without knowing more about the causal functions involved than is usually assumed. We also briefly consider the more philosophical implications of these conclusions.



\section{Overview and related work}
There are typically two qualitatively distinct directions of reasoning involved statistics and causal inference: the first is the \textit{forward} direction, from theories to observable consequences; the second is the \textit{inverse}, or \textit{inferential}, direction, from observable effects to compatible theories. In this work, we explicitly distinguish the spaces of `causal theories' and `observable data', and the directions of the mappings or relations between them. In particular, we establish explicit results connecting conditions on mappings or relations in one direction to properties of mappings or relations in the other. We distinguish existence and stability results. This is important, because while identifiability is formally characterised as the \textit{injectivity} of a mapping or relation in the forward direction, and this in turn implies the \textit{existence} of an inverse mapping or estimator in the reverse direction, this does not guarantee the \textit{continuity} of the inverse. We argue, however, that continuity of the estimator is necessary for estimability, even in principle. That is, estimability fundamentally depends on the continuity properties of the associated estimator, considered at the target population distribution. Without continuity, estimation is impossible with anything less than perfect knowledge of the data distribution, which we would argue violates the concept of ‘able to be estimated from data’.

Simple examples illustrating the gap between identifiability and estimability have been known for a long time. While we will mostly take an inverse problems perspective on these issues, they are also essentially semi-classical results of statistical estimation theory, with strong origins in robust statistics and the functional-analytic point of view. In particular, the `impossibility' results of \citet[][impossibility of estimating the mean]{Bahadur1956-ms}, \citet[][impossibility/possibility of estimating upper/lower bounds for functionals of a density measuring `complexity' of some sort]{Donoho1988-zf} and \citet[][characterisation of `sensitive parameters']{Tibshirani1988-jr} illustrate the basic issues, and are strongly connected to results from robust statistics \citep[see e.g.][]{Hampel2011-cb,Huber2011robust}. Importantly, these impossibility results can be considered as defining \textit{intrinsic} barriers to estimability. In the words of \citet[][p. 186]{Tibshirani1988-jr}, these barriers to estimability are

\begin{quote}
    an intrinsic property of a [quantity], distinct from sampling properties of any estimator.
\end{quote}

Alternatively, these barriers may be thought of providing \textit{ideal bounds} on the possible sampling behaviour of any estimator. Thus, although we frequently use the term `estimator', we are in a sense considering intrinsic `sensitivity' properties of the \textit{parameter} associated with the population value of the estimator. This motivates our choice to work with \textit{functional} estimators, discussed in detail below.

In classical inverse problems theory, and the functional analysis literature more generally, these examples amount to a failure to satisfy the third Hadamard condition \citep{hadamard1902problemes} of a well-posed problem: that of stability of the solution to an inverse problem, \textit{despite satisfaction the second, and possibly first, conditions of uniqueness and existence, respectively}. Here estimability amounts to a problem possessing a `well-posed' solution (i.e. including a stability condition) while identifiability only amounts to requiring uniqueness. 

Tikhonov and others \citep[see e.g.][]{tikhonov1977solutions} developed regularisation theory for ill-posed operator equations to overcome these issues. As noted by \citet{Bickel2006-zi}, this work actually preceded related developments in statistics. In fact, \citet{Donoho1988-zf} also notes that many of his observations were essentially preceded by work on inverse problems in the geophysical literature. Vapnik, Valiant and others \citep[see e.g.][]{vapnik2013nature,vapnik1999overview,Valiant1984-vq,Poggio2004-on,Vito2005-xi} have related and extended some results of (traditionally deterministic) inverse problems to stochastic ill-posed problems and theories of \textit{statistical learnability}. Similar results have also reappeared recently in the context of the instability of Bayesian conditioning \citep{Owhadi2015-sv,Owhadi2015-wy}. \citet{Yu2013-ml} argues for `stability' as a fundamental concept in statistics. There is also a rich tradition of considering the link between identification and estimability concepts in the econometrics/economics literature. In this setting, these issues have been studied under the names \textit{identification at infinity}, \textit{irregular identification}, and/or \textit{ill-posed identification} \citep[see e.g.][for a comprehensive review of identification concepts in econometrics]{Lewbel2016-sk}. The article by \citet{Horowitz2014-us} in particular appears to have a very similar philosophy to ours, discussing economics estimation problems from the perspective of ill-posed inverse problems. \cite{schulman2016stability} provides a realistic example of potential stability issues in causal inference. Finally, our results are very much consistent with those of \citet[][p. 491]{Robins2003-rm} who show, for DAG models, 

\begin{quote}
    the nonexistence of valid, consistent confidence intervals for causal effects and the nonexistence of uniformly consistent point estimators.
\end{quote}
As mentioned above, however, we take a different approach, inspired by the inverse problems literature and an abstract statistical estimation framework that allows for a broad class of `causal modelling' approaches such as those used in the physical sciences. We also focus explicit attention on the separate, formal characterisation of identifiability, its consequences, and the disconnect between identifiability and estimability.





\section{Formalism}
\citet{Evans2002-as} show how much of the machinery of both inverse problems and statistics can subsumed as instances of a general abstract statistical estimation formalism. We largely follow their formalism, though we present a number of basic concepts in terms of elementary category theory as this allows many of the definitions and proofs to be reduced to simple, essentially algebraic, manipulations, while also applying equally to linear and nonlinear problems. This is strongly influenced by the presentation of a `unified theory' of generalised inverses for operator equations defined in various algebraic structures, as developed by \citet{NASHED1976}. We note in particular, as they do, that the \textit{existence} and \textit{uniqueness} results discussed are essentially \textit{algebraic} in nature, while \textit{continuity} results require additional (i.e. topological) structure. Thus we first consider the easier algebraic results before considering continuity. This also allows us to introduce the concept of identifiability in the most general setting possible, before considering what it leaves out. The reader unfamiliar or uncomfortable with our general formalism can, for the most part, simply imagine the mappings involved to be matrices or linear operators, though the results here are not restricted to this case\footnote{Such readers may like to refer to Section \ref{sec:reg} where we give a concrete translation of our algebraic results to the finite dimensional linear regression setting.}. The only presentation of statistical concepts in terms of the concepts of category theory that we are aware of is that given by \cite{McCullagh2002-hm} in an article on the question `what is a statistical model?'. The present article adopts a slightly different usage of the term `model' than \cite{McCullagh2002-hm}, which we hence make explicit below, but apart from some minor terminological differences, it appears that the choice of how to translate basic statistical concepts into categorial concepts is essentially the same in our approach as in theirs.

In what follows we first present our formalism in a general language that, nevertheless, largely follows statistical notation. Then we offer a direct translation of causal identifiability into this setting to show that this language is able to capture the key causal notions of interest. One benefit of our formalism is that it also applies equally well to subjects like inverse problems, where the `causal model' is usually something like a partial differential equation rather than a DAG or structural equation model\footnote{See e.g. \citet{mooij2013ordinary,blom2018generalized,ackley2017compartmental} for interesting work on bridging the gap between DAG-style causal models and (ordinary) differential equation-based causal models.}. 

\subsection{Setting}
We consider an arbitrarily-indexed collection of probability distributions of the form

\begin{equation}
\mathbb{P}_{\theta}(X,Y,...),
\label{eq-family}
\end{equation}
for observable random variables $X, Y,...$ and where each probability distribution instance is labelled by a value of some otherwise arbitrary (and possibly multi-/infinite-dimensional) index $\theta \in \Theta$. We make no assumption about the existence or not of densities for probability measures. As noted by \cite{Evans2002-as} this formalism can accommodate `nonparametric' problems by allowing the indexing space to be arbitrary and possibly infinite-dimensional. We similarly believe that it is important to maintain the presence of an abstract indexing space, even when dealing with infinite-dimensional problems, for reasons that will hopefully become clear.

The family in \eqref{eq-family} above is best thought of as a function, which we call, using the language of inverse problems \citep{Evans2002-as}, the \textit{forward mapping}:

\begin{definition}[Forward mapping]
The function $P$, where
\begin{equation}
P:\Theta \rightarrow \mathcal{P}, \ \theta \mapsto \mathbb{P}_{\theta},
\end{equation}
from the set $\Theta$ of possible values of our index $\theta$ into the set $\mathcal{P}$ containing possible probability measures $\mathbb{P}$ over observable random variables, is called the forward mapping. 
\end{definition}
We will primarily be dealing with the `infinite data' limit, but finite sample data can be imagined to be represented (asymptotically at least) by the associated empirical measure (essentially implying an I.I.D. assumption). Unless otherwise specified, $\mathcal{P}$ is the set of all possible probability distributions over our observables, or a convenient superset of this such as the (linear) space of all finite signed measures.

We emphasise that the distinction between the mapping $P$, the space of distributions $\mathcal{P}$ and specific distributions $\mathbb{P}$ is very important in what follows. \citet{Evans2002-as,Le_Cam2012-cc} often refer to what we call the forward mapping as a `statistical experiment' indexed by `theories' $\theta$. We take a similar approach but will generally prefer the term `forward mapping' as defined above, and use the more concrete term `model' in place of `theory'. Thus the index set $\Theta$ contains what we will call `models'; the forward mapping maps these models to distributions over observables. We will also think of the forward mapping as a `viewing' or `observation' operation when relating our approach to the graphical causal modelling formalism. Thus, while models define theories that lead to distributions over observables, models in our sense are distinct from both the forward mapping and the implied distributions. In causal inference $\Theta$ will contain causal models; in inverse problems $\Theta$ contains e.g. differential equation-based models. To summarise:


\begin{definition}[Models and model space]
Particular values of $\theta$ will be referred to as models, and $\Theta$ as model space. These represent theories that can lead to, but are in general distinct from, distributions over observable variables.
\end{definition}
Although this terminology fits naturally with the causal inference literature, this leads to potential ambiguity since the term `model' (or `statistical model') is often used in the statistical literature \citep[see e.g.][]{McCullagh2002-hm,Bickel2015-fz} to refer to what we call the forward mapping or, at times, the implied measure. As mentioned above, our terminology is consistent with the statistical framework for inverse problems presented by \citet{Evans2002-as}, and is also essentially the same as the terminology used by \citet{Le_Cam2012-cc} if the term `theory' is used in place of `model', and `statistical experiment' is used in place of `forward mapping'. 

One reason for adopting the terminology that we do is that it will allow us to give a direct translation between our framework and the causal inference literature. Importantly, however, unless otherwise specified $\Theta$ can be considered an arbitrary set in our mathematical arguments. This is a simple but important step, closely related to the idea of \citet{wald1950statistical} to allow arbitrary spaces of decisions and hence unify point estimation, interval estimation and testing in statistics under a more general decision theory framework. This influence carries through to our approach via e.g. \citet{Le_Cam2012-cc}.

Again to avoid terminological clashes and to develop our formalism we consider the concept of a `parameter' in more detail. As \citet{Godambe1984-nf} point out, the term `parameter' is often used in at least two distinct ways in statistics: one as an element of an index set (or a function of an element of an index set), and one as a property or feature of a distribution. It is important here to preserve the essence of this distinction. In the present work, again following \citet{Evans2002-as}, functions or functionals of $\theta$, denoted by $q(\theta)$, will be referred to as `parameters' and represent particular properties of interest of a model $\theta$, i.e.

\begin{definition}[Parameters associated with models]
Functions $q$ of models $\theta$ are called parameters and are defined by
\begin{equation}
    q: \Theta \rightarrow \mathcal{V}, \ \theta \mapsto q(\theta)
\end{equation} 
where $\mathcal{V}$ is some appropriate space of values. We will also consider the identity $q(\theta) = \theta$, with $\mathcal{V} = \Theta$, as defining a parameter representing the full model, and will hence also refer to $\theta$ values as parameters. 
\end{definition}
 
In the causal context, discussed further below, parameters like $q(\theta)$ will correspond to so-called `causal queries' of a `causal model' $\theta$. 
 
The development above considers model and parameter space concepts that are only related to observations via the forward mapping. We now define concepts related to mappings in the reverse direction, i.e \textit{estimation} concepts, that take observations, or distributions over observations, as inputs. Again there are subtle terminological issues. In particular, the term `estimator' usually involves finite samples and maps defined from observable random variables $X, Y,...$ to parameter space. However, here we are interested in intrinsic, `infinite data' (asymptotic) barriers to estimation. The robust statistics literature \citep{Hampel1971-jx,Hampel1974-ds,Hampel2011-cb} provides a way of unifying these points of view, by focusing on estimators given by (not necessarily scalar-valued) \textit{functionals} defined on the set of all possible probability distributions (or space of all finite signed measures), and where an estimate is typically obtained by evaluating the estimator at the empirical distribution.
 
For finite samples, many estimators may depend on sample size $n$ in addition to the empirical distribution. However, typically these can be represented asymptotically as functionals of the empirical distribution only \citep{Hampel1974-ds,Hampel2011-cb}. Such estimators can then be viewed, at least asymptotically, as functions defined on a full space of probability distributions. This space can include, for example, a target `true' distribution and relevant neighbourhoods of this. Hence we only consider what we will call \textit{functional estimators}:



\begin{definition}[Functional estimators]
A functional estimator is a function $T$ satisfying
\begin{equation}
T :\mathcal{P} \rightarrow \Theta, \ \mathbb{P} \mapsto \theta.
\end{equation}
\end{definition}
As mentioned, this will include the case where we have, in the usual (somewhat dubious!) statistical terminology, the `true', `full' or `target' population measure and hence -- in the identifiable and exactly-specified case, as we will see -- the `true' population parameter. We will also frequently drop the qualifier `functional' and simply refer to these as `estimators'.

\subsection{Translation of causal concepts}
Here we translate causal concepts into our language (and vice-versa). We assume the reader is familiar with the usual language used in the structural/graphical causal literature \citep[as used in e.g.][]{Pearl2009-qh,Pearl2014-vp}, though the mathematical results given in this article apply to much more general classes of systems and hence are largely independent of many of these details. In brief, we recall the concept of a structural causal model:

\begin{definition}[Structural causal models]
    A (fully-specified) structural causal model $M$ is defined by a set of endogenous variables $V$ and a set of exogenous variables $U$, with a structural equation for each endogenous variable and a joint probability distribution over the exogenous variables.
\end{definition}
For further details, for example on the concept of a structural equation, see \citet{Pearl2009-qh}. In the definition of \citet{Pearl2014-vp} given in the introduction, reference is made to `a set of assumptions $A$'. Following these authors, and the equivalent definitions given by \citet{Pearl2009-qh}, these assumptions are typically, but not always, specified graphically, at least in part, in terms of a causal graph such as a DAG. More generally, the (qualitative/background) causal assumptions can simply be taken to be synonymous with a specification of a \textit{collection} of fully-specified structural causal models, i.e. fully defined models of the `data generating mechanism' with no `free parameters'. $A$ then refers to the common qualitative properties that are satisfied by all models under consideration. These properties may or may not be explicitly characterised by some representation such as a graph (e.g. a DAG), though frequently such graphical representations are employed for this purpose. This is formalised as:


\begin{definition}[Causal classes]
Fully-specified causal models $M_1, M_2$ are said to belong to a class $\mathcal{M}$ \textit{iff} they satisfy the set of qualitative properties defined by the set of assumptions $A$. 
\end{definition}
Such assumptions may include, for example, the requirement that each model is compatible with a given directed acyclic graph (DAG) $G$. In general we may simply take the set of assumptions $A$ as synonymous with the specification of a set $\mathcal{M}$ of fully-specified causal models. Furthermore we define:

\begin{definition}[Causal queries]
A causal query $Q: \mathcal{M} \rightarrow \mathcal{V}, \ M \mapsto Q(M)$ is a function mapping fully-specified models in a common class $\mathcal{M}$ to some value space $\mathcal{V}$.
\end{definition}
Thus in the (more abstract) language of the present work, models $M_1, M_2$ correspond to models $\theta_1, \theta_2$, the causal class $\mathcal{M}$ corresponds to the model space $\Theta$ to which $\theta_1, \theta_2$ belong, and causal queries $Q(M)$ correspond to parameters $q(\theta)$. We also translate the function\footnote{Which is apparently not explicitly named in the causal setting, and hence is only implicitly distinguished from its value $P(M)$.} $P$ which maps any fully-specified model $M$ to its probability distribution $P(M)$ into our forward mapping $P$. We summarise this in:

\begin{definition}[Translation of basic causal concepts into our abstract framework]

\begin{equation}
\begin{aligned}
\mathcal{M}\ \ &\leftrightarrow\ \ \Theta\\
M_1, M_2 \in \mathcal{M}\ \ &\leftrightarrow\ \ \theta_1, \theta_2 \in \Theta\\
Q(M) \ \ &\leftrightarrow\ \ q(\theta)\\
P: \mathcal{M} \rightarrow \mathcal{P}, \ M \mapsto P(M)  \ \ &\leftrightarrow\ \ P:\Theta \rightarrow \mathcal{P}, \ \theta \mapsto P(\theta)
\end{aligned}
\end{equation}
\end{definition}
This translation implies that any results concerning what we call models and/or parameters in our more abstract general setting also translate directly to the causal setting. 

In the following example we give a more concrete illustration of the translation, which we will return to again after we have discussed identifiability of parameters $q(\theta)$ and their estimators in more detail.

\begin{example}[Structural and graphical causal models]
    Here we give a simple example of translating a few basic structural/graphical causal modelling ideas into our formalism. A schematic of this example is given in Figure \ref{fig:graphical-translation-I}. Here the model space $\Theta$ contains fully-specified structural causal models, i.e. a given (abstract) indexing vector $\theta$ corresponds to a specification of all information required to determine a structural model. For example $\theta_1$ may represent a structural model with $Z := f_Z(X)$, $Y := f_Y(X,Z)$ and a probability distribution over the exogenous variables $X$ given. Often additional disturbance or `error' terms are given as implicit exogenous variables, for example we may write e.g. $Y := f_Y(X,Z,U_Y)$ with $U_Y$ an additional error term. Hence, given a $\theta$, we have structural equations for all endogenous variables of $\theta$ and distributions over all exogenous (external) variables of $\theta$. Graphically, exogenous variables are simply those with no arrows pointing into them and endogenous variables are those with at least one incoming arrow, though sometimes additional (exogenous) error terms are not shown explicitly.
    
    In this example we allow models with two qualitatively different directions of causality for the relationship between $X$ and $Z$ (as well as the case of no causal relationship). For example we may wish to try to use data to decide which causal direction is correct -- unfortunately this is impossible at this level of detail, as it can be shown that the two models $\theta_1$ and $\theta_2$ have the same probabilistic implications (we discuss how to determine these implications below). This also means that we have chosen to allow $X$ and $Z$ to be thought of as either endogenous or exogenous, depending on which graph is considered. We thus suppose that we have available both structural equations for $X$ and $Z$ and exogenous distributions over $X$ and $Z$, and each model makes use of either a structural equation or a distribution for each, depending on which graph is considered. We assume a set of consistency conditions between any two fully-specified models. For example, here we assume that the two structural equations $Z := f_Z(X)$ and $X := f_X(Z)$ have the same solution sets, and that the structural equation $Y:= f_Y(X,Z)$ is the same for each model shown. The probability distributions for any exogenous variables shared between models are assumed to be the same in both models (and if additional error terms are included we take all exogenous variables to be independent). In other problems we may wish to distinguish between different structural equations that have the same set of dependent and independent variables, or allow more complex exogenous distributions. Furthermore, in the general case $\theta$ may be infinite-dimensional and capture quite arbitrary nonparametric structural equations. The qualitative assumptions determined by specification of the set $\Theta$ can also be quite arbitrary in general.
    
    We interpret the forward mapping $P$ as an instruction to `view' $Y|X$, meaning we simply `record' the value of this conditional probability for the given graph. Given a fully-specified structural model, the mapping `view$_{Y|X}$' can be operationalised by conditioning on $X$ while marginalising over $Z$ and any remaining exogenous variables (the probabilistic interpretation of graphical models in terms of graph factorisations can also be used). In the translation given here, we have chosen to represent probabilistic conditioning and marginalisation as part of the forward mapping/viewing operation, rather than consider graphs in which variables are explicitly conditioned on \citep[as in the `submodels' in e.g.][p. 204]{Pearl2009-qh}, as we believe this provides an even clearer separation between probabilistic `viewing' and causal `doing' operations: all probabilistic operations are considered `viewing' operations. This is one possible translation, however, and minor variants are possible. Importantly, in reality we only have access to finite samples from (and/or the corresponding empirical distribution for) these models. For now we stick to the idealised setting, but return to this crucial issue when we discuss stability/continuity considerations and ill-posed problems. 
    
    The query $q$ here is to `do' $X$, which is interpreted in essentially the same way as in the graphical causal literature: remove all arrows into $X$. Since we interpret conditioning as part of the `viewing' operation, however, here the `do' step is simply to make $X$ an exogenous variable. We relate this to the underlying structural equations by replacing the structural equation for $X$ by the exogenous distribution over $X$ (again, assumed to be available as part of the background assumptions). In this context the space $\mathcal{V}$ is a subset of the original space $\Theta$, i.e. `do$_X$' is a projection operation, mapping the space $\Theta$ into itself. For example, we require, for all $\theta \in \Theta$, do$_{X}(\theta) = \theta'$ for some $\theta'$ $\in \Theta$. This is another motivation for making conditioning part of the forward mapping, as otherwise we would have to consider a slightly modified value space $\mathcal{V}$, for example one in which the value of $X$ was also conditioned on. We note, however, that identifiability of queries $q(\theta)$ can be analysed even if the value space is different to (or not a subset of) $\Theta$. We return to identifiability of queries $q(\theta)$ later.
    
    Bringing the above together, the combination of our version of `do$_X$' with the instruction to `view $Y$ conditional on a particular $X$ value' completes the translation of the usual specification of `do$(X=x)$' in the graphical/strucutral literature. This means the basic translation of what is usually denoted $P_M(Y=y|do(X=x))$ in the causal literature is, here, the composite operation view$_{Y=y\vert X=x} \circ$ do$_{X}(\theta) = \mathbb{P}_{do_X(\theta)}(Y=y|X=x)$, where $M \leftrightarrow \theta$ is a fully specified model. Again, this is one possible translation, and others may wish to combine conditioning on variable values and removal of arrows into one operation.
\end{example}

\newpage 
\begin{figure}
\begin{center}
    \includegraphics[scale=0.8]{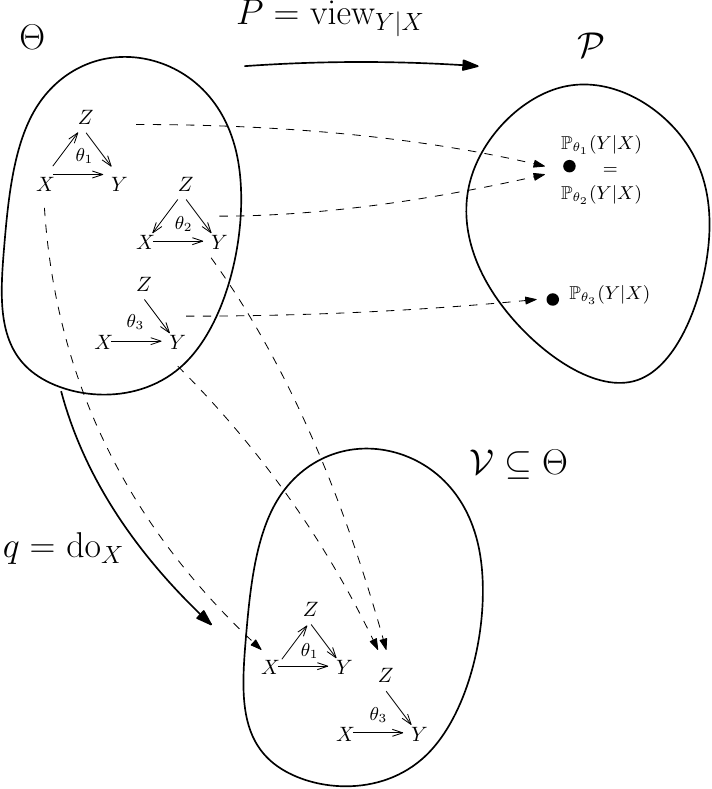} 
\end{center}
\caption{Illustration of translating structural/graphical causal models into our abstract formalism. The labels on the graphs here are abstract, and represent distinct, fully-specified structural causal models. We have allowed the direction of causality between $X$ and $Z$ to be unknown.}
\label{fig:graphical-translation-I}
\end{figure}
\clearpage

\newpage
\section{Identifiability and the existence of functional estimators}
Here we demonstrate the first step in what would seem to be the direct connection between identifiability and estimability: we show that identifiability implies the existence of Fisher-consistent\footnote{The idea of Fisher consistency is supposed to capture the idea that if these estimators could be evaluated `at the true population', then we would obtain our target quantity, and hence is closely related to the intuition of identifiability. This section demonstrates this tight connection. While we are aware of previous work investigating the links between identifiability and \textit{asymptotic consistency} \citep[where the former is necessary but not sufficient for the latter, see e.g.][]{Gabrielsen1978-wb,Martin2002-it}, we are unaware of explicit analyses of the link between identifiability and \textit{Fisher consistency}, as considered here. Based on the equivalence that we demonstrate in what follows, Fisher consistency appears to be the most natural estimation-oriented counterpart to identifiability, though asymptotic consistency is closely related to the topic of the rest of this article, i.e. \textit{stable} estimation.} functional estimators that can be written as mappings from probability distributions back to the indexing (model or parameter) space. In short, we first show that identifiability of the forward mapping implies the existence of the inverse mapping. Only after this essentially algebraic result do we show what breaks the apparent equivalence between identifiability and estimability: discontinuity of the inverse (estimator) in general.

\subsection{Definitions: Identifiability and Fisher consistent estimators}
We first consider the identifiability of full models (or `full parameters') here; the definition of identifiability of parameters $q(\theta)$ is given in a following subsection, where we show how our present arguments extend naturally to parameters. Our definition of identifiability for full models is given as follows:

\begin{definition}[Identifiability of domains, models, and mappings]
\label{def:identifiability}
The domain $\Theta$ of a forward mapping $P:\Theta \rightarrow \mathcal{P}$ is called identifiable iff $P$ is injective, i.e. a 1-1 mapping. We will also use the term identifiable to refer to the models in a common identifiable domain and to the injective mapping $P$ itself.
\end{definition}

An injective mapping corresponds to the concept of a monomorphism in category theory \citep[see e.g.][]{Lawvere2003-db,Lawvere2009-qo,geroch2015mathematical}, and can be given an essentially algebraic characterisation, which we give below\footnote{We adopt the perspective of elementary category theory here for precisely this reason: it provides a general algebra of abstract mappings -- e.g. not necessarily linear, or even functions between sets -- which means that many results that are often presented in special cases can be shown to hold in much more general settings.}. Hence, while injectivity of a mapping $P$ is usually expressed as

\begin{equation}
P(\theta_1) = P(\theta_2) \implies \theta_1 = \theta_2,
\end{equation}
a more general (or more algebraic/categorial) characterisation is 

\begin{definition}[Injectivity]
\label{def:injectivity}
$P$ is an injective mapping iff it is left-cancellable, i.e.

\begin{equation}
P \circ G = P \circ H \implies G = H
\end{equation}
for all maps $G$ and $H$ with common domain, and with ranges in $P$'s domain. 
\end{definition}

In category theory, elements $\theta \in \Theta$ can be thought of as special types of maps $\mathbf{1} \overset{\theta}{\rightarrow} \Theta$, where $\mathbf{1}$ is the singleton set, hence $P(\theta)$ as $P \circ \theta $, which informally indicates the equivalence of our definition to the usual definition \citep[a formal proof can be found in e.g.][]{Lawvere2003-db}. Other ideas from category theory that lead to useful cancellation properties include the following. Firstly, an algebraic/categorial characterisation of onto/surjective maps, or epimorphisms in the category of sets and mappings, is that they are right-cancellable. Secondly, an important property of left-inverses -- also called retractions -- and of right-inverses -- also called sections -- is that they are epimorphisms and monomorphims, respectively, and hence are right-cancellable and left-cancellable, respectively. The converse doesn't hold in general categories, but does hold given the axiom of choice (one form of which can be stated as: `all epimorphisms have sections'). Again, in the usual category of sets and mappings the terms `epimorphism' and `monomorphism' can be replaced by the more familiar terms `surjective' and `injective', respectively, and the axiom of choice is usually assumed to hold.

Next we give a categorial characterisation of Fisher consistency. First we recall that an estimator is called Fisher consistent if \citep{Huber2011robust,Hampel2011-cb}:

\begin{equation}
T(\mathbb{P}_{\theta}) = \theta
\end{equation}
for all $\theta$. In terms of the forward mapping this is

\begin{equation}
T\circ P(\theta) = \theta
\end{equation}
for all $\theta$. This can be summarised as saying

\begin{definition}[Fisher consistency]
An estimator $T$ is Fisher consistent for a forward mapping $P$ iff
\begin{equation}
T\circ P =1_{\Theta}
\end{equation}
where $1_{\Theta}$ is the identity map on $\Theta$. 
\end{definition}

This means that $T$ is a left-inverse (retraction) for $P$ and $P$ is a right-inverse (section) for $T$. Although $T$ is defined on all of $\mathcal{P}$, or at least a dense subset of it, Fisher consistency only really makes sense for distributions in the range of $P$, i.e. we only need to check that we can correctly \textit{recover} a parameter when we know the data was generated from a model with that parameter value. 

\subsection{The existence of a Fisher consistent estimator is equivalent to identifiability (given the axiom of choice)}

Here we give an elementary proof of the following result: 

\begin{theorem} The existence of a Fisher consistent estimator is equivalent to identifiability of the forward mapping, assuming the axiom of choice.
\end{theorem}

\begin{proof} Firstly, suppose that there exists a Fisher-consistent estimator, i.e.

\begin{equation}
T\circ P =1_{\Theta}.
\end{equation}
Now suppose we have $P(\theta_1) = P(\theta_2)$. In this case

\begin{equation}
T\circ P(\theta_1) = \theta_1
\end{equation}
and

\begin{equation}
T\circ P(\theta_2) = \theta_2
\end{equation}
by Fisher consistency. But $T\circ P(\theta_1) = T(P(\theta_1)) = T(P(\theta_2)) = T\circ P(\theta_2)$, since $P(\theta_1) = P(\theta_2)$ by assumption and $T$ is a function, and thus $\theta_1 = \theta_2$, as required to show identifiability.

For the other direction, we suppose that the forward mapping is identifiable, i.e. is a monomorphism (1-1). We will demonstrate the existence of a Fisher-consistent estimator. In order to do so, we will use a condition that turns out to be equivalent to the axiom of choice. This does not hold in arbitrary categories, or even in arbitrary toposes, but is usually considered to hold in the usual category of sets and mappings. Here it amounts to the assumption of the existence of a generalised inverse for all maps in the category, known elsewhere \citep[see e.g.][]{NASHED1976} as the assumption that all maps are (von Neumann) regular or quasi-invertible\footnote{While a seemingly technical matter, the need to assume this to establish the existence of an estimator given identifiability points to one source of the difficulty in establishing a stable connection between estimability and identifiability: it allows for an existence result to hold while providing no guarantee on the further properties (e.g. stability) of the object or on any procedure of construction of the object.}. This form of the axiom is likely most familiar from the characterisation of generalised inverses in linear algebra, but the key idea also carries over to nonlinear maps.

The most convenient form of the required axiom can be expressed as \citep{Lawvere2003-db}: 

\begin{definition}[Axiom of choice.]
For all maps $P$ there exists a $G$ such that

\begin{equation}
P\circ G\circ P = P,
\label{eq-inner}
\end{equation}
where $G$ is called an inner inverse to $P$. 
\end{definition}

As discussed by \cite{NASHED1976}, inner inverses are not generally outer inverses as well, where a map $G$ is called an outer inverse for $P$ if it satisfies:

\begin{equation}
G\circ P\circ G = G.
\label{eq-outer}
\end{equation}
Generalised inverses in the usual sense are required to be both inner and outer inverses. Given an inner inverse, however, a map which is both an inner inverse and an outer inverse can \textit{always} be constructed from an inner inverse $G$ by taking $G^\dagger = G\circ P\circ G$. Hence for simplicity we will assume that this construction has been carried out and that we are working with a map that is also an outer inverse, i.e. is a proper generalised inverse.

Now, if, in \eqref{eq-inner}, we simply take $T=G$ and take $P$ to be our usual foward mapping, we guarantee the existence of a $T$ satisfying

\begin{equation}
P\circ T\circ P = P.
\end{equation}
From the identifiability assumption we know that $P$ is 1-1 and hence left-cancellable, which implies

\begin{equation}
T\circ P = 1_\Theta
\end{equation}
as required to show Fisher consistency.
\end{proof}

Thus we conclude: identifiability is equivalent to the existence of a (Fisher consistent) estimator (in categories with the axiom of choice). 

\subsection{Note on idempotent maps}
While $T\circ P$ is the identity in the above setting, $P \circ T$ is not the identity in general, but \textit{is} an idempotent map\footnote{i.e. a general projection mapping satisfying $A^2 = A$.} since

\begin{equation}
(P\circ T) \circ (P \circ T) = (P\circ T \circ P) \circ T = P \circ T.
\end{equation}
In linear regression problems in statistics this is often called the \textit{hat matrix} \citep{Huber2011robust,Hoaglin1978-vt}, while in inverse problems it is sometimes called the \textit{data resolution operator} \citep{aster2013parameter}: it maps the actual data distribution to a `smoothed' or estimated data distribution. This map will be the identity only if $P$ is surjective (onto); in most statistical estimation problems it is not surjective and a \textit{reduction of data} \citep[see e.g.][]{fisher1990statistical} occurs. 

To help the reader grasp the relatively abstract discussion so far, we give a brief illustration of how these ideas work in the classical statistical context of linear regression next, before continuing the translation of these abstract ideas to the \textit{causal} setting.

\subsection{Illustration of basic concepts}\label{sec:reg}
Here we give an elementary illustration of the above ideas in terms of the linear regression problem.

\begin{example}[Linear regression]
In the basic linear regression context we wish to `solve' the linear problem

\begin{equation}
    A\theta = y
\end{equation}
where $A$ is an $m \times n$ real matrix mapping parameters to data, $\theta \in \mathbb{R}^n$ is the parameter (model) vector, and $y \in \mathbb{R}^m$ is the data vector and is assumed to be given\footnote{Here we are following the standard linear regression setting where parameters are mapped to data vectors, rather than mapped to data distributions. Either setting is covered by the general formalism. See also \citet{Evans2002-as} for the relation between these settings.}. The goal is to estimate (i.e. solve for) $\theta$ given $A$ and $y$. Here injectivity of $A$, which is equivalent to identifiability of $\theta$ as shown above, means that the columns of $A$ are linearly independent. We do not assume that $A$ is surjective -- in fact in regression problems it will not be. Thus to `invert' $A$, for a general $y$ not necessarily in the range of $A$, we need to find a generalised inverse, here denoted by $A^\dagger$, such that it satisfies our inner inverse condition, i.e.

\begin{equation}
    A A^\dagger A = A.
\end{equation}
From basic linear algebra we know that linear independence of $A$'s columns means that, while $A$ is not invertible in general, $A^TA$ is positive definite and invertible, where $A^T$ is the transpose of $A$. This means that the following is well-defined as a possible choice for $A^\dagger$:

\begin{equation}
    A^\dagger = (A^TA)^{-1}A^T.
\end{equation}
Substituting this into the left-hand side of our inner inverse condition gives:

\begin{equation}
    A (A^TA)^{-1}A^T A = A I = A
\end{equation}
which verifies that $A^\dagger$ is in fact an inner inverse. Furthermore, because $A$ is injective/identifiable, then by the general argument above $A^\dagger$ should also a left inverse of $A$. This can be directly verified:

\begin{equation}
    A^\dagger A = (A^TA)^{-1}A^T A = I.
\end{equation}
$A^\dagger$ is not a right inverse in general, however, i.e. $AA^\dagger \neq I$ in general. Instead $AA^\dagger$ is a projection (idempotent) matrix since

\begin{equation}
    (AA^\dagger)(AA^\dagger) = (A(A^TA)^{-1}A^T)(A(A^TA)^{-1}A^T) = A (A^TA)^{-1}A^T = AA^\dagger
\end{equation}
Putting this together, we can take $A^\dagger$ as our estimator (denoted by $T$ in the general arguments of the previous section), with estimated \textit{model/parameter}

\begin{equation}
    \hat{\theta} = (A^TA)^{-1}A^T y = A^\dagger y
\end{equation}
and estimated/smoothed \textit{data}

\begin{equation}
    \hat{y} = A(A^TA)^{-1}A^T y = A A^\dagger y = \hat{A}y
\end{equation}
where $\hat{A} = AA^\dagger$ defines the so-called hat matrix. Our estimator is Fisher-consistent because $A^\dagger$ is a left inverse of $A$ and hence, if our data was generated by $\theta$ i.e. $y=A\theta$, then

\begin{equation}
     A^\dagger y = A^\dagger A \theta = (A^TA)^{-1}A^T A = I \theta = \theta
\end{equation}
for any $\theta$, i.e. our estimator recovers the true value. In fact, the above holds true when $y = A\theta + e$, for $e$ orthogonal to the columns of $A$, i.e. $A^Te = 0$, meaning that $A\theta$ represents the conditional expectation of $y$.
\end{example}

All of the above is standard statistics, and essentially just linear algebra. Nothing is said about e.g. the stability of the above -- as is well known, such regression estimates can be extremely unstable in the presence of outliers \citep{Hampel2011-cb,Huber2011robust}. The point of the more abstract formalism adopted in the rest of this article is that the same essentially \textit{algebraic} concepts carry over to the general \textit{nonlinear} setting as well. Furthermore, as we will show next, because these results only rely on general algebraic concepts they can also be carried over to the \textit{causal} inference setting using an appropriate interpretation of the relevant domains, codomains and mappings. Again, however, these results neglect stability and/or topological considerations. The final parts of the present article will thus be concerned with these topics.

\subsection{Parameters and causal queries}
Here we extend our previous proof to the case of parameters $q(\theta)$, and then show that our concept of parameter identifiability translates, in the special case of causal models, to causal identifiability. We then consider Fisher consistency for parameters. We will again treat elements as special types of maps, e.g. for a parameter value $v \in \mathcal{V}$ given by $v = q(\theta)$, we can also treat $v$ as a mapping $\mathbf{1} \overset{v}{\rightarrow} \mathcal{V}$, where $\mathbf{1}$ is the singleton set. Hence an expression like $s(v)$ can equivalently be read as $s \circ v$, and vice-versa.

\subsubsection{Identifiability for parameters}
The usual definition of identifiability of a parameter function $q(\theta)$ \citep[see e.g.][]{Evans2002-as} is:

\begin{definition}[Identifiability of parameters -- standard]
\label{def:ident-q-stand}
A parameter $q: \Theta \rightarrow \mathcal{V}$ is identifiable with respect to $P$ iff for all $\theta_1 \neq \theta_2$ 
\begin{equation}
    P(\theta_1) = P(\theta_2) \implies q(\theta_1) = q(\theta_2)
\end{equation}
or, equivalently, iff for all $\theta_1 \neq \theta_2$ 

\begin{equation}
    q(\theta_1) \neq q(\theta_2) \implies P(\theta_1) \neq P(\theta_2).
\end{equation}
\end{definition}

This definition is the same as the definition of causal identifiability given in the introduction, in the special case of the causal interpretation of our more abstract framework. The definition is also an injectivity requirement but now corresponds to requiring the \textit{relation} from $q(\theta)$ values to $P(\theta)$ values to be injective. That is, this relation is no longer \textit{univalent} (i.e. no longer a function), which means that there is not necessarily a unique or well-defined \textit{generative model} from $q$ values to observable data. Similarly, \citet{McCullagh2002-hm} provides an essentially set-theoretical, but equivalent, definition of (sub-) parameter identifiability. Hence a general abstract characterisation of sub-parameter identifiability and its implications could be pursued in relational, or bicategory, terms. We adopt a slightly different formalisation, however, continuing to use ideas from the category of sets and mappings to relate identifiability of sub-parameters to the original generative model $P$. We believe this offers a useful conceptual viewpoint, and one that motivates constructive computer implementations of estimation procedures using functions. The definition amounts to considering all possible `representatives' of the inverse images of points $v$ in the range space $\mathcal{R}(q)$ of $q$ in \citet{McCullagh2002-hm}'s definition. We show that our new definition is equivalent to the standard definition. Our definition is as follows:

\begin{definition}[Identifiability of parameters -- choice of representatives]
\label{def:ident-q-gen}
A parameter $q: \Theta \rightarrow \mathcal{V}$ is identifiable with respect to $P$ iff 

\begin{equation}
\begin{aligned}
\forall s&:\mathcal{R}(q)\rightarrow \Theta \text{ such that } q\circ s = 1_{\mathcal{V}}, \forall v_1, v_2 \in \mathcal{V}\\
&v_1 \neq v_2 \implies P \circ s \circ v_1 \neq  P \circ s \circ v_2
\end{aligned}
\end{equation}
or, equivalently, iff 

\begin{equation}
\begin{aligned}
\forall s&:\mathcal{R}(q)\rightarrow \Theta \text{ such that } q\circ s = 1_{\mathcal{V}}, \forall v_1, v_2 \in \mathcal{V}\\
&P \circ s \circ v_1 =  P \circ s \circ v_2  \implies v_1 = v_2.
\end{aligned}
\end{equation}
\end{definition}

As mentioned above, one interpretation of the above is that we are working with arbitrary choices of `representative' parameters in $q$-induced equivalence classes, where, for $v = q(\theta)$, representatives are defined by $\theta^* = s\circ v = s\circ q (\theta)$ for any $s$. Note that $(s\circ q) \circ (s \circ q) = s\circ (q \circ s) \circ q = s \circ q$ since $q\circ s = 1_{\mathcal{V}}$ and hence $s\circ q$ is idempotent, i.e. is a general projection that maps equivalence class members to representatives. The images of these representatives under $P$ are $P(\theta^*) = P(s\circ v) = P(s\circ q (\theta))$. It follows that the above can be interpreted as saying that the composite function $P\circ s$ is a 1-1 function for all sections $s$ of $q$, i.e. for all choices of representatives. To verify that nothing is lost, we will prove the equivalence to the standard definition.

\begin{theorem}
The parameter $q$ is identifiable, with respect to the forward mapping $P$, according to the standard definition iff it is identifiable, with respect to $P$, according to our alternative definition.
\end{theorem}
\begin{proof}
We will prove each direction by contrapositive, i.e. we will show non-identifiability according to one definition implies non-identifiability according to the other. 

Firstly, suppose that $q$ is not identifiable according to our alternative definition. Then there exists an $s:\mathcal{R}(q)\rightarrow \Theta, \text{ such that } q\circ s = 1_{\mathcal{V}}$, and $v_1, v_2 \in \mathcal{V}$, such that $P \circ s \circ v_1 =  P \circ s \circ v_2$ but $s \circ v_1 \neq s \circ v_2$. Define $\theta_1' = s \circ v_1$ and $\theta_2' = s \circ v_2$. Then $q(\theta_1') = q\circ s\circ v_1 = v_1$ (by the key property of $s$) and $q(\theta_2') = q\circ s\circ v_2 = v_2$. Hence $q(\theta_1') \neq q(\theta_2')$ but $P(\theta_1') = P \circ s \circ v_1 = P \circ s \circ v_2 = P(\theta_2')$, which demonstrates non-identifiabilty according to the usual definition. 

Next, suppose that $q$ is non-identifiable according to the usual definition. Then there exists $\theta_1$ and $\theta_2$ for which $P (\theta_1) =  P(\theta_2)$ but $q(\theta_1) \neq q(\theta_2)$. We simply need to show that there always exists an $s$ for which the previous construction works `in reverse'. Since $q(\theta_1) \neq q(\theta_2)$ we know that the range of $q$, i.e. the domain of $s$, contains at least two distinct points, $v_1 = q(\theta_1)$ and $v_2 = q(\theta_2)$, $v_1 \neq v_2$. We will construct an $s$, starting from these two points, and extend it to the rest of the range of $q$ (if any). Define $s$ on these two points by $s (v_1) = \theta_1$ and $s (v_2) = \theta_2$. Then clearly $q \circ s$ and $s \circ q$ are the identity mappings when restricting attention to $\{v_1,v_2\}$ and $\{\theta_1,\theta_2\}$, respectively, i.e. $s$ defines a bijection when restricted to these two sets. Next consider $q'$ mapping $\Theta \setminus q^{-1}\{v_1,v_2\}$ to $\mathcal{R}(q)\setminus\{v_1,v_2\}$. This is clearly surjective since $q$ is surjective on $\mathcal{R}(q)$. Thus, by the axiom of choice, there exists an injective mapping $s'$ defined on $\mathcal{R}(q)\setminus\{v_1,v_2\}$ such that $q'\circ s' = 1_{\mathcal{R}(q)\setminus\{v_1,v_2\}}$. Defining $s$ to be equal to this $s'$ over $\mathcal{R}(q)\setminus\{v_1,v_2\}$, and defined as above for $\{v_1,v_2\}$, thus ensures the existence of a mapping satisfying $q\circ s = 1_{\mathcal{V}}$. Furthermore, we have $P\circ s \circ v_1 = P(\theta_1) = P (\theta_2) = P\circ s \circ v_2$, while we know $v_1 \neq v_2$ by assumption. Thus $q$ is non-identifiable according to our definition.
\end{proof}

To make clearer the (equivalence of the) above definitions of identifiability we give a simple example of when a parameter, $q$, is not identifiable with respect to a given forward mapping $P$. We will, for simplicity, consider a forward mapping $P$ which is vector-valued rather than measure-valued\footnote{This example can be directly related to the probabilistic case when the model is based on an additive error assumption, or a location model. As shown in \cite{Evans2002-as}, a parameter in this setting is identifiable \textit{iff} it is identifiable using just the associated deterministic component of the forward mapping.}.

\begin{example}[A non-identifiable mapping and parameter]
\label{ex:q}
Here we consider a simple example of a non-identifiable mapping and parameter. For simplicity, here we take all of our domains and ranges to be (real-valued) $n$-spaces and the associated mappings $q:\mathbb{R}^2\rightarrow\mathbb{R}$ and $P:\mathbb{R}^2\rightarrow \mathbb{R}$ as linear transformations. Hence we have $\Theta=\mathbb{R}^2$, $\mathcal{V}=\mathbb{R}$ and $\mathcal{P}=\mathbb{R}$. Then, for any $\theta\in\Theta$ we can write $\theta=(x,y)^T$ where $x,y\in\mathbb{R}$, and define the parameter of interest by $v = q(\theta)=q\theta=x$, i.e., $q=(1, 0)$. Furthermore, we define $P$ by $P(\theta)=P\theta=x+y$, i.e., $P=(1,1)$. We now show that $q$ is not identifiable with respect to $P$ in terms of Definitions \ref{def:ident-q-stand} and \ref{def:ident-q-gen}. To do this, let $\theta_1=(x_*,y_*)^T$ and $\theta_2=(y_*,x_*)^T$, for some $x_*,y_*\in\mathbb{R}$  with $x_*\neq y_*$.

Considering Definition \ref{def:ident-q-stand}, note that $P(\theta_1)=x_*+y_*=y_*+x_*=P(\theta_2)$, while $q(\theta_1)=x_*\neq y_*=q(\theta_2)$. Hence $q$ is not identifiable with respect to $P$ according to the standard Definition \ref{def:ident-q-stand}. Next, considering Definition \ref{def:ident-q-gen}, let $s$ be defined by:


$$
s(v)=
\begin{cases}
(v,y_*)^T, & v=x_*\\
(v,x_*)^T, & v=y_*\\
(v,0)^T, & v\notin\{x_*,y_*\}.
\end{cases}
$$
This can be thought of as, for example, choosing $\theta_1 = (x_*,y_*)$ and $\theta_2 = (y_*,x_*)$ as the representatives for $v_1 = x_*$ and $v_2 = y_*$, respectively, while choosing the Moore-Penrose pseudo-inverse, $s=(1,0)^T=q^\dagger$, for $v\in\mathbb{R}\setminus\{x_*,y_*\}$. It is easily verified that for all $v\in\mathbb{R}$, $q\circ s(v)=v$, i.e. $q\circ s= 1_\mathbb{R}=1$. Furthermore, we have $ v_1\neq v_2$, since $x_* \neq y_*$, while $P\circ s\circ v_1=P(\theta_1)=x_*+y_*=y_*+x_*=P(\theta_2)=P\circ s\circ v_2$. Therefore $q$ is not identifiable with respect to $P$ according to Definition \ref{def:ident-q-gen}.
\end{example}
A schematic of the above example is given in Figure \ref{fig:non-id}.

\newpage 
\begin{figure}
\begin{center}
 \begin{tikzpicture}[scale=0.8]
     \draw (0,0) -- (4,0);
     \draw (6,0) -- (10,0);
     \draw (8,-2) -- (8,2);
     \draw (12,0) -- (16,0);
     \node (a) at (6, 2) {};
     \node (b) at (2, 1) {};
     \draw[->] (a)  to [out=150,in=60] (b);
     \draw (5,2.8) node[left] {$q$};
     \node (c) at (10, 2) {};
     \node (d) at (14, 1) {};
     \draw[->] (c)  to [out=30,in=120] (d);
     \draw (11,2.7) node[right] {$P$};
     \node (e) at (2.5, 1) {};
     \node (f) at (6, 1.5) {};
     \draw[->] (e)  to [out=60,in=150] (f);
     \draw (5,1.75) node[left] {$s$};
     
     \draw[ultra thick] (6.2,0) -- (9.8,0);
     \node[inner sep=0pt,minimum size=3pt,circle,fill=black,label=below:{$v_1$}] (v1) at (1,0) {};
     \node[inner sep=0pt,minimum size=3pt,circle,fill=black,label=below:{$v_2$}] (v2) at (3,0) {};
     \node[inner sep=0pt,minimum size=3pt,circle,fill=black,label=left:{$\theta_1$}] (theta1) at (7,1) {};
     \node[inner sep=0pt,minimum size=3pt,circle,fill=black,label=right:{$\theta_2$}] (theta2) at (9,-1) {};
     \node[inner sep=0pt,minimum size=3pt,circle,fill=black,label=below:{$P(\theta_1)=P(\theta_2)$}] (Ptheta) at (14,0) {};
     \node[circle,draw=black, fill=white, inner sep=0pt,minimum size=3pt] (b) at (7,0) {};
     \node[circle,draw=black, fill=white, inner sep=0pt,minimum size=3pt] (b) at (9,0) {};
   
     \draw (15,1) node[above right] {$\mathcal{P}=\mathbb{R}$};    
     \draw (1,1) node[above left] {$\mathcal{V}=\mathbb{R}$};  
     \draw (8,2.5) node[above] {$\Theta=\mathbb{R}^2$};  
     \draw (2,2.9) node[below] {$q\circ s=1_{\mathcal{V}}$};   
\end{tikzpicture}
\end{center}
\caption{Illustration of Example \ref{ex:q}, showing a non-identifiable parameter $q$, relative to a forward mapping $P$. The section $s$ `picks out' the representatives of the non-identifiable $q$ values, while for simplicity is equal to the minimum norm representatives in the other cases. The values of $s$ are indicated by the dark lines/points in the middle diagram.}
\label{fig:non-id}
\end{figure}

We can also illustrate an example of non-identifiability by re-considering our first example, updated in Figure \ref{fig:graphical-translation-II}.
\begin{example}
This continues Example 1 where we translated a few key causal ideas into our framework in the context of a simple example. We can capture the key intuition behind identifiability of parameters for structural and graphical causal models by calling the `choice of representatives' (section) mappings (denoted by $s$ above) `undo' operations. In particular, while `do$_X$' means all arrows into $X$ are erased, i.e. the structural equation for $X$ is replaced by an exogenous probability distribution for $X$, `undo$_X$' means potentially adding arrows into $X$, i.e. potentially replacing the exogenous distribution for $X$ by a structural equation for $X$. Figure \ref{fig:graphical-translation-II} provides a corresponding update of Figure \ref{fig:graphical-translation-I} with an example choice of undo operation added.

Here we assume that there are only three possible links between $X$ and $Z$: that given by $\theta_1$, $\theta_2$ or $\theta_3$. Hence for $\theta_1$ we cannot add any arrows, while for $\theta_3$ we have two options -- add the same structural equation as is present in $\theta_2$ (and hence become the $\theta_2$ model) or keep $X$ as exogenous (and hence stay the $\theta_3$ model). In general we may allow multiple possible structural equations for $X$ but here we assume only these cases for simplicitly. The choice of an undo operation compatible with a given do operation is of course not unique; however, any choice of $s$ must satisfy $q \circ s = 1_{\mathcal{V}}$, i.e. do$_X \circ$ undo$_X = 1_{\mathcal{V}}$ for all possible choices of undo$_X$. That is, either no arrows are added or any arrows added would be removed by a subsequent `do' operation. This is as appears in Definition \ref{def:ident-q-gen}, where we also have multiple choices of representatives, or undo operations, to consider. Definition \ref{def:ident-q-gen} then translates as `a query $q$ = do$_X$ is identifiable iff for all possible undo operations compatible with $q$, view$_{Y|X}\circ undo_X = P\circ s$ is a 1-1 function'. 

On the other hand, if there exists an undo mapping that allows two models to be mapped to new models with additional arrows into $X$ such that they have same observational implications, do$_X$ is not identifiable. This is this case in our example: by not knowing whether $X$ causes $Z$ or $Z$ causes $X$, we are unable to determine the query do$_X$. This is well-known in the causal literature \citep{Pearl2009-qh}.

Typically causal identifiability analysis is carried out algebraically. However, we can now formulate this (somewhat non-rigorously) as an optimisation problem as follows. First, we note that given a $\theta'$ imagined to have resulted from a do$_X(\theta)$ operation for some $\theta$, we can form a correspondence between `choices of representatives' and undo-do-value pairs via $\theta_1 \leftrightarrow (s_1,\theta')$, $\theta_2 \leftrightarrow (s_2,\theta')$ and so on. One of these will be the original $\theta$, but there may be multiple possible alternatives in general. We can then say that do$_X$ is non-identifiable if, for some values $\theta'_1 \neq \theta'_2$ of do$_X(\theta)$ we have that $\inf_{s_1,s_2} d(P(s_1,\theta'_1),P(s_2,\theta'_2)) =0$ for an appropriate choice of distance $d$ between probability distributions. Otherwise the query is identifiable. Here we have adopted the abuse of notation $P(s,\theta') := P(s(\theta'))$, based on the 1-1 correspondence mentioned above. This optimisation problem is closely related to checking for non-identifiability by checking for a `flat profile' in $\theta'$ when maximising over `nuisance parameters' $s$ in the sense of profile likelihood (but here in the `infinite data' limit), and is also closely related to minimum distance estimation more generally. This idea also applies when the value space $\mathcal{V}$ is not equal to/is not a subset of $\Theta$, and is a useful practical approach to identifiability analysis in complex models, including those not expressed as structural causal models and/or when the target query is not itself a fully-defined model. 

\end{example}

\begin{figure}[h]
\begin{center}
    \includegraphics[scale=0.8]{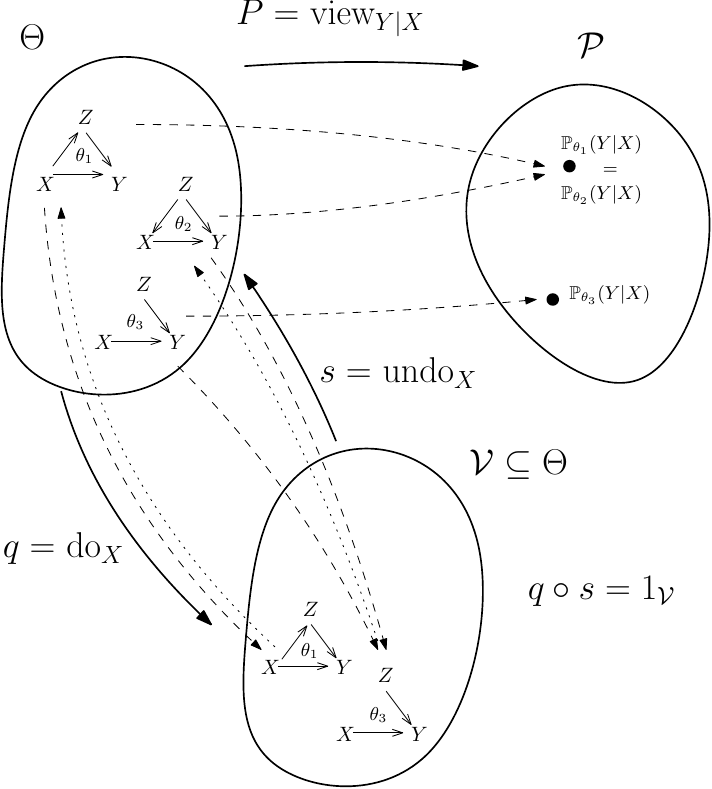} 
\end{center}
\caption{Illustration of translating structural/graphical causal models into our abstract formalism. This is an updated version of Figure \ref{fig:graphical-translation-I}. The labels on the graphs here are abstract, and represent distinct, fully-specified structural causal models. We have allowed the direction of causality between $X$ and $Z$ to be unknown. This leads to non-identifiability of do$_X$, and we have added one example `undo' mapping that detects this non-identifiability; in general there are multiple undo mappings (right inverses/sections) for a given do operation: given some $q$, any mapping $s$ satisfying $q\circ s = 1_{\mathcal{V}}$ is a valid choice, and our characterisation of identifiability requires all such possible choices to be considered.}
\label{fig:graphical-translation-II}
\end{figure}

Next we will extend the key theorem relating identifiability and the existence of Fisher-consistent estimators from the previous section. 

\subsubsection{Fisher consistency for parameters}
First we define estimators of parameters $q$, and Fisher consistency for these parameters:

\begin{definition}[Fisher consistency for parameter estimators]
A functional estimator for a parameter $q: \Theta \rightarrow \mathcal{V}$ is a function $t:\mathcal{P} \rightarrow \mathcal{V}$. A functional estimator $t$ is said to be Fisher-consistent for $q$ \textit{iff} $t$ satisfies:

\begin{equation}
\begin{aligned}
\forall s&:\mathcal{R}(q)\rightarrow \Theta \text{ such that } q\circ s = 1_{\mathcal{V}},\\
& t\circ P \circ s  = 1_{\mathcal{V}} .
\end{aligned}
\end{equation}

\end{definition}
Importantly, $t$ does not depend on the choice of $s$, i.e. it does not depend on the choice of $\theta$ representatives for the $q$ values. As noted in the previous section, although $t$ is defined over all $\mathcal{P}$, or a dense subset of this, Fisher consistency only really makes sense over the range of $P$, i.e. we only need to check that we can correctly \textit{recover} a parameter when we know the data was generated from a model with that parameter value. 

\begin{theorem}
A parameter $q$ is identifiable iff there exists a Fisher consistent estimator for $q$.
\end{theorem}
\begin{proof}
Firstly, suppose that $q$ is identifiable. Then for any $s$ with $q \circ s = 1_{\mathcal{V}}$, $P\circ s$ is a 1-1 function on $\mathcal{V}$. By the previous section, this implies the existence of a $t$, possibly depending on the choice of $s$, such that $t\circ (P\circ s) = 1_{\mathcal{V}}$. We will show that $t:\mathcal{P} \rightarrow \mathcal{V}$ is in fact independent of $s$, at least when considered as an estimator $t'$ defined over distributions in the range of $P$, i.e. as $t':\mathcal{R}(P) \rightarrow \mathcal{V}$ (distributions outside of the range of $P$ cannot make $t$ Fisher inconsistent, since there is no `true' parameter to recover). 

To do this, suppose we have two pairs $s_1, t_1$ and $s_2, t_2$ satisfying the above and consider, for arbitrary $v \in \mathcal{V}$, an element $\mathbb{P}_{v}$ of $\mathcal{R}(P) \subseteq \mathcal{P}$, obtained by two different choices of $s$, $s_1$ and $s_2$, respectively. That is, assume that we have

\begin{equation}
\begin{aligned}
\mathbb{P}_{v} = P \circ s_1 (v) = P \circ s_2 (v),
\end{aligned}
\end{equation}
regardless of whether $s_1 = s_2$. We show that $t_1(\mathbb{P}_{v}) = t_2(\mathbb{P}_{v}) = v$, i.e. that they define the same function $t':\mathcal{R}(P) \rightarrow \mathcal{V}$ from distributions in the range of the forward mapping to parameter values. By Fisher consistency, we know that for all $v$ we have

\begin{equation}
    t_1 (P \circ s_1 (v)) = t_2 (P \circ s_2 (v)) =v,
\end{equation}
and hence

\begin{equation}
\begin{aligned}
&t_1(\mathbb{P}_v) = t_2(\mathbb{P}_v) = v.
\end{aligned}
\end{equation}
Thus we have a unique $t':\mathcal{R}(P) \rightarrow v$ independent of the choice of $s$. Clearly if $q$ is an identifiable parameter then we can extend $t'$ to a Fisher consistent estimator $t :\mathcal{P} \rightarrow v$ defined on all of $\mathcal{P}$, since $t$ can take any value for distributions not in the range of $P$ without sacrificing Fisher consistency.

Finally we prove that if a Fisher consistent estimator for a parameter $q$ exists, then $q$ is identifiable. Suppose that a Fisher consistent estimator, $t$, exists for $q$. Now suppose that for arbitrary $s$ such that $q \circ s = 1_{\mathcal{V}}$, and for arbitrary $\theta_1, \theta_2$, we have $P\circ s\circ q(\theta_1) = P \circ s \circ q(\theta_2)$. Then we have, by Fisher consistency,

\begin{equation}
    t \circ P\circ s\circ q(\theta_1) = q(\theta_1)
\end{equation}
and

\begin{equation}
    t \circ P\circ s\circ q(\theta_2) = q(\theta_2).
\end{equation}
But $t \circ P\circ s\circ q(\theta_1) = t(P\circ s\circ q(\theta_1)) = t(P\circ s\circ q(\theta_2)) = t \circ P\circ s\circ q(\theta_1)$, since $P\circ s\circ q(\theta_1) = P \circ s \circ q(\theta_2)$ by assumption and $t$ is a function. Thus $q(\theta_1) = q(\theta_2)$ and $s\circ q(\theta_1) = s\circ q(\theta_2)$, and hence we have identifiability.
\end{proof}

\section{Continuity considerations: the break between identifiability and estimability}
So far we have shown, under very minimal assumptions, that identifiability leads to the existence of a suitable Fisher-consistent estimator. Now we show that, despite appearances, this mere existence result is not sufficient to guarantee what we will call estimability, even in principle. The key (missing) requirement is a stability condition: we should still be able to estimate the quantity of interest with just an `infinitesimal' amount of uncertainty about the population probability distribution. This requires more structure than identifiability and Fisher consistency, and is hence more difficult to discuss in general terms. Without this, however, estimation -- in any sense connected to the real world -- is impossible. Furthermore, stability conditions lead to a need to reconsider the `existence' part of the above argument, in terms of the domain of definition of the generalised inverse/estimator. Thus we will see that identifiability does not guarantee estimability, and there are many cases where any estimator of a given identified quantity has unbounded sensitivity and hence is inestimable in principle. 

\subsection{Well-posed and ill-posed problems}
A number of concepts have appeared in the inverse problems and statistical literature, all related in some way to the realisation that, in addition to existence and/or uniqueness, stability, continuity and sensitivity considerations are required. We briefly recap a few general concepts here, before giving specific examples and references in the following subsections.

Firstly, in the inverse problems literature, and in the mathematical literature more generally, it is common to follow \citet{hadamard1902problemes} and call a problem \textit{well-posed} if the following three semi-formal criteria hold \citep[see also][]{engl1996regularization,tikhonov1977solutions}:

\begin{itemize}
    \item For all admissible problem data, a solution \textit{exists}
    \item For all admissible problem data, the solution is \textit{unique}
    \item The solution depends \textit{continuously} on the problem data, i.e. is \textit{stable}.
\end{itemize}

The problem then is called \textit{ill-posed} if at least one of the above conditions does not hold. To make these conditions fully formal requires specific definitions to be given to e.g. the admissible problem data of interest, the relevant topology for continuity, and so on, but these capture the basic idea over a range of cases. \citet{engl1996regularization,tikhonov1977solutions} provide good overviews of classical, mostly deterministic, inverse problems theory. Other references that we draw on extensively include \citet{NASHED1976} and \citet{groetsch1977generalized}.

In general failure to satisfy either of the first two criteria can essentially be solved in an algebraic sense by the concept of a generalised inverse, discussed earlier, which amounts to solving the problem `as well as can be'. Aside from the condition of being an inner and outer inverse, a unique generalised inverse requires a choice of up to two idempotent maps (projections) in general; a choice of these can be shown to define a unique algebraic generalised inverse \citep{NASHED1976}. In the case of an injective map, we saw above that failure of surjectivity leads to the idempotent but non-identity map $P\circ T$ on data distributions. Since the `model resolution' operator is the identity in this case, only the data resolution projection operator needs to be chosen to give a unique solution. In statistics a common choice leads to satisfying the data in the least squares sense. Similarly, failure of injectivity (i.e. failure of identifiability) can be `solved' by considering the associated idempotent (but now non-identity) mapping $T\circ P$ defined on the \textit{parameter/model} space. In statistics and inverse problems, a common choice to resolve non-uniqueness is to take the model with the smallest norm, and so on. For a general forward mapping $P$, which is neither an injective nor surjective mapping, one must choose two idempotent mappings -- one on the data distribution space, one on the parameter space. Combined with the inner and outer inverse conditions, this leads to a unique algebraic generalised inverse \citep[again, see][]{NASHED1976}.

\subsection{Continuity, stability, and sensitivity}
\subsubsection{Ideas from inverse problems and functional analysis}
The third Hadamard condition given above is typically more difficult and subtle to satisfy, and is \textit{not} resolved by the usual, essentially algebraic, concept of a generalised inverse. Furthermore, as discussed by \citet{NASHED1976}, once continuity conditions come into play, a distinction between algebraic generalised inverses and topological generalised inverses must be introduced to handle the surprisingly subtle interplay between existence, uniqueness, and continuity. In particular, much more attention must be given to the relevant domains and codomains of definition, and a distinction between algebraic subspace complements and topological subspace complements becomes important.

First we give the standard definitions of algebraic and topological complementary subspaces. Again we largely follow \citet{NASHED1976} and, in places, \citet{groetsch1977generalized}. 

\begin{definition}[Algebraic and topological complementary subspaces]
Given a vector space $\mathcal{A}$, we say that two subspaces of $\mathcal{A}$, $\mathcal{A}_1$ and $\mathcal{A}_2$, are complementary subspaces, or algebraic complements, in $\mathcal{A}$ iff 

\begin{equation}
    \mathcal{A} = \mathcal{A}_1 \oplus \mathcal{A}_2
\end{equation}
where $\oplus$ denotes the algebraic direct sum, i.e. if every $a \in \mathcal{A}$ can be uniquely written as the sum $a_1 + a_2$ of an element $a_1 \in \mathcal{A}_1$ and an element $a_2 \in \mathcal{A}_2$. These are called topological complementary subspaces, or topological complements, if, in addition to being algebraic complements, any associated linear projector (idempotent map) $P$ with $\mathcal{R}(P) = \mathcal{A}_1$ and $\mathcal{R}(I-P) = \mathcal{A}_2$ is continuous. Or, equivalently, if these subspaces are both closed.
\end{definition}
The existence of topological complements for any closed subspace is not guaranteed in arbitrary Banach spaces, but is guaranteed in any Hilbert space. In Banach spaces we must in general add the supposition that topological complements exist, and we will assume this in what follows.

Now we consider the relevant domains and codomains, given $P:\Theta \rightarrow \mathcal{P}$. In particular we suppose that $P$'s codomain can be written as 

\begin{equation}
    \mathcal{P} = \overline{\mathcal{R}(P)} \oplus \mathcal{C}
\end{equation}
where $\mathcal{C}$ is the topological complement of the \textit{closure} of the range $R(P)$ of $P$. From the definition of topological complementary subspaces, this complement is always closed. In a Hilbert space we have $\mathcal{C} = \mathcal{R}(P)^\perp$ and so

\begin{equation}
    \mathcal{P} = \overline{\mathcal{R}(P)} \oplus \mathcal{R}(P)^\perp.
\end{equation}
On the other hand, we take the natural domain of definition $\mathcal{D}(T)$ of our generalised inverse (estimator) $T$ to be 

\begin{equation}
    \mathcal{D}(T) = \mathcal{R}(P) \oplus \mathcal{C},
\end{equation}
where $\mathcal{C}$ is the same as before (i.e. the topological complement of the closure of the range). Again in a Hilbert space we have $\mathcal{C} = \mathcal{R}(P)^\perp$ and so

\begin{equation}
    \mathcal{D}(T) = \mathcal{R}(P) \oplus \mathcal{R}(P)^\perp.
\end{equation}
Note that unless $\mathcal{R}(P) = \overline{\mathcal{R}(P)}$, i.e. unless the range of $P$ is closed, then $\mathcal{D}(T) \neq \mathcal{P}$. This seemingly subtle distinction will play an important role in establishing the continuity, or not, of the generalised inverse.

\subsubsection{Conditions for continuous/discontinuous generalised inverses}
Here we consider various special cases of the general requirement that (the generalised inverse) $T$ be continuous. The linear theory is relatively standard and well-developed, while the nonlinear case less so. Nevertheless, we can also give various general conditions for when a nonlinear problem is ill- or well-posed.

Firstly, consider the case where the forward mapping $P$ is an everywhere-defined, linear and continuous operator which maps between two Hilbert spaces\footnote{Many of these results can also be extended to closed, densely-defined but possibly unbounded linear operators with domains and codomains that are Banach spaces, or at least Banach spaces that have topological complements.}. Many of the following results are essentially consequences of the famous Open Mapping Theorem (Banach and Schauder), extended to generalised inverses. The standard theorem can be stated in our setting as:

\begin{theorem}[Open Mapping Theorem]
Consider a linear operator $P:\Theta \rightarrow \mathcal{P}$. If $P$ has closed range $\mathcal{R}(P)$, then the image $P[\mathcal{O}]$ of any open set $\mathcal{O}\subset \Theta$ is an open set in $\mathcal{R}(P)$. Similarly, the image of any closed subset is closed. If $P$ is also injective and surjective, then $P^{-1}:\mathcal{P} \rightarrow \Theta$ exists and is continuous (bounded).
\end{theorem}
Proofs may be found in textbooks on functional analysis. This result implies the Closed Graph Theorem:

\begin{theorem}[Closed Graph Theorem]
Let $P:\Theta \rightarrow \mathcal{P}$ be a graph-closed operator, i.e. have closed graph where the graph of an operator is given by

\begin{equation}
    \mathcal{G}(P) = \{(\theta,P(\theta)\ \vert\ \theta \in \Theta \}.
\end{equation}
Then $P$ is continuous.
\end{theorem}

These theorems extend to similar results concerning generalised inverses, i.e. Fisher-consistent estimators, e.g.:

\begin{theorem}
The generalised inverse (estimator) $T:\mathcal{R}(P) \oplus \mathcal{R}(P)^\perp \subseteq \mathcal{P} \rightarrow \Theta$, associated with the linear, continuous forward operator $P:\Theta \rightarrow \mathcal{P}$, is a linear operator which is bounded (continuous) if and only if the range of $P$, $\mathcal{R}(P)$, is closed. In this case $\mathcal{R}(P) \oplus \mathcal{R}(P)^\perp = \mathcal{P}$ and hence $T$ is defined on all of $\mathcal{P}$. Furthermore, since $\mathcal{R}(P\circ T) = \mathcal{R}(P)$, we have that $\mathcal{R}(P)$ is closed, and hence $T$ is continuous, iff the range $\mathcal{R}(P\circ T)$ of the idempotent operator $P\circ T$ is closed.
\end{theorem}
\begin{proof}
A sketch of a proof goes as follows. Recall that a function is continuous \textit{iff} its inverse set image of any closed set is also closed. The inverse set image of the closed set $\Theta$ under the generalised inverse is $\mathcal{R}(P) \oplus \mathcal{R}(P)^\perp \subseteq \mathcal{P}$, which cannot be closed if $\mathcal{R}(P)$ is non-closed. Hence continuity implies closed range. The converse, that closed range implies continuous generalised inverse, follow directly from the Closed Graph Theorem.
\end{proof}

This theorem is stated for arbitrary Hilbert space domains and codomains in e.g. \citep{NASHED1976,engl1996regularization,groetsch1977generalized}; here we have simply translated this to the special case where the domain of $P$ is $\Theta$ and the codomain of $P$ is $\mathcal{P}$. Importantly, Hilbert spaces are often very natural settings for the domains and codomains in inverse problems and statistical problems \citep{Small2011-cl,Evans2002-as,vapnik2013nature,groetsch1977generalized,engl1996regularization,Ramsay2005-qw}, so this is no real restriction, though more general results in e.g. Banach spaces also exist. 

The following \textit{stability} condition is often more convenient to work with than the closed range condition \citep{drabek2009methods,groetsch1977generalized}:

\begin{theorem}
If $P:\Theta \rightarrow \mathcal{P}$ is a linear, continuous, injective (identifiable) forward operator between Hilbert (or Banach) spaces, then its range $\mathcal{R}(P)$ is closed iff there is a positive constant $c$ such that for all $\theta \in \Theta$:

\begin{equation}
    \Vert P(\theta) \Vert \geq c\Vert \theta \Vert.
\end{equation}
\end{theorem}
See e.g. \citep{drabek2009methods,groetsch1977generalized} for proofs. In such a case we say the forward mapping $P$ is \textit{bounded away from zero}. Importantly, one can have operators that are invertible and yet not bounded away from zero -- in such cases one has a sequence of singular values with zero as a limit point.

We can also see how the above gives a stability condition on the \textit{estimator/estimates} in the sense that it means there exists a $c > 0$ such that 

\begin{equation}
    \Vert \theta_1 - \theta_2 \Vert \leq \frac{1}{c}\Vert P(\theta_1) -P(\theta_2) \Vert
\end{equation}
which follows directly from linearity and taking $\theta = \theta_1 - \theta_2$. That is, defining $\mathbb{P}_1 = P(\theta_1)$ and $\mathbb{P}_2 = P(\theta_2)$, we have

\begin{equation}
    \Vert T(\mathbb{P}_1) - T(\mathbb{P}_2) \Vert \leq \frac{1}{c}\Vert \mathbb{P}_1 -\mathbb{P}_2 \Vert
\end{equation}
for a Fisher consistent estimator $T$, which is hence a bounded/continuous linear operator. On the other hand, since $P$ is bounded, we also have, for some $b>0$ and for all $\theta$,

\begin{equation}
    \Vert P(\theta) \Vert \leq b\Vert \theta \Vert.
\end{equation}
This means that we can write the continuity of $T$ as

\begin{equation}
    \frac{\Vert T(\mathbb{P}_1) - T(\mathbb{P}_2) \Vert}{\Vert T(\mathbb{P}_2)\Vert} \leq \frac{b}{c}\frac{\Vert \mathbb{P}_1 -\mathbb{P}_2 \Vert}{\Vert \mathbb{P}_2\Vert}.
\end{equation}
For finite-dimensional linear problems, the range is in fact always closed and hence all finite-dimensional linear problems are \textit{formally} stable in the above sense. As is well-known in the applied mathematics literature, however, finite-dimensional problems can exhibit a form of instability essentially equivalent to ill-posedness, called \textit{ill-conditioning}. To see this, consider the stability theorem for finite-dimensional inverse problems in the form:

\begin{theorem}
If $P:\Theta \rightarrow \mathcal{P}$ is a linear, continuous, injective (identifiable) and finite-dimensional forward operator then 

\begin{equation}
    \frac{\Vert T(\mathbb{P}_1) - T(\mathbb{P}_2) \Vert}{\Vert T(\mathbb{P}_2)\Vert} \leq \frac{\sigma_{\text{max}}}{\sigma_{\text{min}}}\frac{\Vert \mathbb{P}_1 -\mathbb{P}_2 \Vert}{\Vert \mathbb{P}_2\Vert},
\end{equation}
\end{theorem}
where $\sigma_{\text{max}}$ and $\sigma_{\text{min}}$ are the maximum and minimum singular values of $P$, respectively \citep[see e.g.][for an elementary derivation]{aster2013parameter}. The ratio $\frac{\sigma_{\text{max}}}{\sigma_{\text{min}}}$ is known as the \textit{condition number}. Thus in such problems, while they are technically well-posed, they can be ill-posed for any achievable precision, e.g. to machine zero, or at least exhibit essentially uncontrollable instability. This leads to the following semi-formal definition:

\begin{definition}[Ill-conditioned inverse problems]
A finite-dimensional linear inverse problem is known as ill-conditioned (c.f. ill-posed in the infinite-dimensional case) if the condition number of the forward mapping is unacceptably large.
\end{definition}

Finally, we briefly consider nonlinear problems. Linear operators have the property that they are either continuous everywhere or nowhere, and that continuity and boundedness are equivalent. This makes continuity easy to characterise in terms of the (mapping from the domain to) range only. Nonlinear operators do not have this property, however, so are much more difficult to deal with in general. We will not attempt to survey the literature on nonlinear functional analysis here, but note that the following (uni-directional) result does still hold (we assume for convenience here that $\Theta$ and $\mathcal{P}$ are Hilbert spaces):

\begin{theorem}
The generalised inverse (estimator) $T:\mathcal{R}(P) \oplus \mathcal{R}(P)^\perp \subseteq \mathcal{P} \rightarrow \Theta$, associated with the (possibly nonlinear) continuous forward operator $P:\Theta \rightarrow \mathcal{P}$, has at least one point of discontinuity if the range of $P$, $\mathcal{R}(P)$, is non-closed.
\end{theorem}
\begin{proof}
A function is continuous \textit{iff} its inverse set image of any closed set is also closed. The inverse set image of the closed set $\Theta$ under the generalised inverse is $\mathcal{R}(P) \oplus \mathcal{R}(P)^\perp \subseteq \mathcal{P}$, which cannot be closed if $\mathcal{R}(P)$ is non-closed. 
\end{proof}
Furthermore, we have the following key result which relies on the idea of restricting models to \textit{compact} subsets\footnote{Recall that the general topological characterisation of a compact set is that every open cover has a finite subcover. In $\mathbb{R}^n$ the concept of a compact subset is the same as a closed and bounded subset, but this does not hold in more general spaces.} of the full parameter space \citep{tikhonov1977solutions,zeidler1985nonlinear}:
\begin{theorem}
Suppose $P:M \rightarrow \mathcal{R}(P)$ is a continuous, injective forward mapping defined on a compact subset $M\subseteq \Theta$. Then $P^{-1}:\mathcal{R}(P)\rightarrow M$ is continuous.
\end{theorem}
For a proof and natural extensions to generalised inverses see \citep{tikhonov1977solutions}.

The concept of compactness plays a key role in nonlinear functional analysis, and provides the means to extend many results that hold for finite-dimensional spaces to infinite-dimensional spaces and to express key approximation methods \citep{zeidler1995applied,zeidler1985nonlinear}. Compactness also underlies the concept of \textit{regularisation}, in which an ill-posed problem defined on a non-compact domain is replaced by a \textit{sequence of nested restricted problems} with the solutions of each problem restricted to one of a sequence of nested, compact subsets of the domain, $M_1 \subset M_2 ... \subset M_n ...$, and where the closure of their union equals $\Theta$, i.e. $\overline{\cup_{n=1}^{\infty}M_n} = \Theta$. Each solution provides a \textit{stable approximation} of the solution to the full problem, and the art of regularisation lies in choosing the trade-off between the goodness of approximation and stability. In reality, this choice must also take into account ill-conditioning in addition to formal ill-posedness, as discussed above. 

A good discussion of classical regularisation theory, stochastic ill-posed problems and the link between compactness and the concept of `capacity control' in statistical learning theory is given by \citet{vapnik2013nature}. Instability of nonlinear estimators has also been considered in the standard statistical literature, which we consider next.

\subsubsection{Ideas from statistics}
Essentially the same continuity concepts as discussed above have arisen in the statistical literature. In contrast to the inverse problems literature, however, much of this literature tends to assume the existence of (Fisher) consistent estimators \textit{a priori}, and hence assume (typically implicitly) identifiability. The main focus is then the further properties of these estimators, such as continuity or robustness. Again, perhaps counter to some intuitions, we find the implication that mere existence of a (Fisher) consistent estimator, and hence identifiability, does not imply estimability even in principle.

An early example is the illustration by \citet{Bahadur1956-ms} of the impossibility of fully nonparametric estimation of the mean. As shown there, it is impossible to obtain a nontrivial confidence interval for the mean, when considered as a functional defined over the set of all distributions for which this exists and is finite. Subsequent work of particular relevance here is that of \citet{Donoho1988-zf} on \textit{one-sided inference} and \cite{Tibshirani1988-jr} on \textit{sensitive parameters}. These nicely illustrate the general concepts underlying the impossibility result of \citet{Bahadur1956-ms}, and are also closely related to work in robust statistics \citep[see e.g.][]{Huber2011robust,Hampel2011-cb}, in particular the general notion of \textit{qualitative robustness} and the related (local) robustness measures derived from the \textit{influence function}, both introduced by \citet{Hampel1971-jx}. From a slightly different but related direction, \citet{vapnik2013nature} discusses the key role that classical inverse problems theory played in the development of his and colleagues' work in the theory of statistical machine learning. More recent work on stability and statistical machine learning, with strong connections to inverse problems theory, includes that by \citet{Poggio2004-on,Vito2005-xi}.

First we briefly consider the concept of sensitive parameters, introduced by \citet{Tibshirani1988-jr}, and building on \citet{Bahadur1956-ms}. They use the term functional parameter for what we call the estimator $T$ mapping (ideal) distributions back to parameter/indexing space. We will hence call these sensitive estimators; following \cite{Tibshirani1988-jr} we emphasise, however, that these are functions of \textit{population} distributions, and are not finite sample quantities.

\citet{Tibshirani1988-jr} define a sensitive parameter $T$, in our terminology called a sensitive estimator, for $\mathcal{P}$ a complete, separable metric space with distance $\delta$, and $\Theta =\mathcal{R}(T) \subseteq \mathbb{R}^n$ as:

\begin{definition}[Sensitive estimator]
\label{def:sensitive}
An estimator $T:\mathcal{P} \rightarrow \Theta$ is called sensitive with respect to metric $\delta$ iff for all $\mathbb{P} \in \mathcal{P}, \theta \in \Theta, \epsilon > 0$, there exists $\mathbb{Q} \in \mathcal{P}$ such that

\begin{equation}
    \delta(\mathbb{P},\mathbb{Q}) \leq \epsilon \text{ and } T(\mathbb{Q}) = \theta.
\end{equation}

\end{definition}
This means that, given a probability distribution $\mathbb{P} \in \mathcal{P}$ we can always find a distribution $\mathbb{Q} \in \mathcal{P}$ that is arbitrarily close to $\mathbb{P}$, in terms of the probability distribution distance $\delta$, and yet will give an arbitrary estimator value. In particular, given $\mathbb{P}$, an arbitrarily close $\mathbb{Q}$ can always be found that nevertheless makes $T(\mathbb{Q})$ arbitrarily different to $T(\mathbb{P})$. \citet{Tibshirani1988-jr} primarily take $\delta$ to be the total-variation distance, but also consider the Prohorov and contamination distances\footnote{As they note, this latter distance is not a metric.}. They show that sensitivity with respect to the contamination distance implies sensitivity with respect to the others. Similarly to \citet{Bahadur1956-ms}, a notable example of a sensitive parameter is the mean.

Hampel considered stability earlier, in the context of robust statistics and finite samples. \citet{Tibshirani1988-jr} show, however, that we can consider the $n \rightarrow \infty$ limit and obtain essentially the same conclusions. In particular, they show by considering the limiting case of the definition of qualitative robustness given by \citet{Hampel1971-jx}, that qualitative robustness for a real-valued estimator implies the estimator is continuous with respect to the Euclidean distance between the estimator values\footnote{They further note, however, that their concept of sensitivity is in a sense more extreme than discontinuity in some cases.}. Additional discussion of the relationships between qualitatively robust sequences of estimators, continuous sequences of estimators, and continuity of functional estimators is given by \citet{Hampel2011-cb,Huber2011robust}. For now we introduce a key \textit{heuristic} tool from robust statistics -- the \textit{influence function} \citep{Huber2011robust,Hampel2011-cb}, which essentially amounts to the Gateaux derivative of $T$ at $F$ in the direction of the point mass $\Delta_y$ at $y$ (though it exists under weaker conditions).

\begin{definition}[Influence function]
The influence function is defined by

\begin{equation}
    \text{\normalfont IF}(y;T,F) = \lim_{\epsilon \downarrow 0}\frac{T((1-\epsilon)F + \epsilon \Delta_y) - T(F)}{\epsilon}.
\end{equation}
where $\Delta_y$ is the point mass at $y$.
\end{definition}
It can be shown that the asymptotic variance of the estimator is, under some regularity conditions, given by

\begin{equation}
    V(T,F) = \int \text{IF}(y;T,F)^2dF(y)
\end{equation}
Smaller variance corresponds to greater (asymptotic) efficiency, and one use of the influence function is hence to help construct efficient estimators. There is, however, a \textit{trade-off between efficiency and robustness} to perturbations or sensitivity, e.g. infinitesimal contamination. As we will see, the mean is maximally efficient for location estimates under the normal family, but is also maximally sensitive to distributional contamination. The influence function also aids in these considerations. A measure of robustness in terms of the influence function is the so-called \textit{gross-error sensitivity} of $T$ at $F$:

\begin{definition}[Gross-error sensitivity]
The gross-error sensitivity is defined by:
\begin{equation}
    \gamma^* = \sup \vert \text{\normalfont IF}(y;T,F) \vert,
\end{equation}
i.e. the maximal value of the influence function, where this is taken over all $y$ for which this exists.
\end{definition}
This measures the worst influence that an infinitesimal contamination can have, and can be considered as giving \textit{an upper bound on the asymptotic bias} \citep{Hampel2011-cb}. Unbounded influence functions thus give unbounded gross-error sensitivity and unbounded asymptotic bias. As discussed by \cite{Tibshirani1988-jr}, functionals with unbounded gross error sensitivity are sensitive in their sense.

The key message for our purposes is simply that there are strong connections between the statistical notions of `sensitivity' and `robustness', and continuity of estimators\footnote{These connections are particularly strong when identifiability is assumed or generalised inverses are used. E.g. the median is robust and insensitive but is technically discontinuous unless it is uniquely defined e.g. as the minimum value for which the cumulative distribution is equal to 0.5. Identifiability and/or the concept of the generalised inverse naturally lead to such an additional condition and hence continuity of the median.}.

Finally, and again following \citet{Bahadur1956-ms}, \citet{Donoho1988-zf} gives conditions for nonlinear statistical functionals to be \textit{inestimable}, at least in the two-sided sense. In particular, Donoho considers the `truly' or `strongly' nonparametric setting, which he describes in intuitive terms as\footnote{We use our notation of $\mathcal{P}$ for families of distributions and $\mathbb{P}$ for a probability distribution to avoid conflicts of terminology with the rest of the text.}:

\begin{quote}
    Intuitively, a truly nonparametric family of distributions has the property that, when it contains a distribution $\mathbb{P}$ it also contains all other distributions which cannot be reliably distinguished from $\mathbb{P}$ at a given sample size based on any empirical test...Following this line of reasoning, we arrive at the requirement that, for $\mathcal{P}$ to be nonparametric, it should contain at least a small $\tau$ neighbourhood around essentially every point.
\end{quote}
He demonstrates, also by introducing the topological concepts necessary to characterise continuity and related notions, that a number of functionals of distributions are badly discontinuous\footnote{Which he characterises in terms of a dense graph condition, closely related to the idea of sensitivity in \cite{Tibshirani1988-jr}, as they note.} in this setting. This includes the mean, as well as several measures of complexity such as the number of modes of a density. Interestingly, while the mean cannot be given either lower or upper bounds, he demonstrates that many measures of complexity \textit{can} be given lower bounds, just not upper bounds -- e.g. that one can say a distribution must be `at least this complex', but cannot rule out more complex distributions based on any empirical test. Again the important point for our purposes is that additional, e.g. topological, concepts and continuity must be considered in order to characterise \textit{estimability}, even in the identifiable scenario.

\section{So, what is estimability?}
Which the above in mind, we can give a tentative and semi-formal characterisation of \textit{estimability}, with respect to our current setting:

\begin{definition}[Estimability]
\label{def:estimability}
Given a forward mapping $P:\Theta \rightarrow \mathcal{P}$, between models $\theta \in \Theta$ and observable distributions $\mathbb{P} \in \mathcal{P}$, a parameter $q: \Theta \rightarrow \mathcal{V}$ is estimable if it is identifiable and, in addition, the implied Fisher-consistent estimator $t:\mathcal{P} \rightarrow \mathcal{V}$ is sufficiently continuous. 
\end{definition}
This includes the case where, for example, one would say a formally well-posed but ill-conditioned estimator leads to inestimability. We have also only stated this condition as a sufficient condition, as one may even consider \textit{dropping the strict identifiability requirement} and consider instead e.g. estimation subject to bounded asymptotic bias \citep[see e.g.][]{Evans2002-as}.

\section{Examples of identifiable but inestimable quantities}
Here we consider, using special cases of our more general setting, simple examples of problems that arise in statistics, causal inference and inverse problems. In each subsection we first formulate the basic problems, then we consider stability. We largely follow \citet{vapnik2013nature} in converting statistical questions to operator equations, and \cite{Hampel2011-cb,Huber2011robust} in our discussion of influence functions and robustness. Similar examples have appeared in the econometrics literature \citep[see e.g.][and references therein]{Khan2010-bh,Lewbel2016-sk,Escanciano2018-xo,Horowitz2014-us}.

\subsection{Basic estimation problems}
\subsubsection{Conditional density estimation}
Consider the causal model described in Figure \ref{fig:dag}, where $U$ is a (potentially unmeasured) mediator for the $X\longrightarrow Y$ relationship.

\begin{figure}[h]
\begin{center}
\begin{tikzcd} 
&U \arrow[rd]&\\
X \arrow[ru]\arrow[rr] && Y
\end{tikzcd} 
\end{center}
\label{fig:dag}
\caption{Causal DAG where $U$ is a (potentially unmeasured) mediator between $X$ and $Y$.}
\end{figure}
In this case we have

\begin{equation}
    p(y | do(x)) = p(y |x)
\end{equation}
and hence our causal estimation problem reduces to a statistical estimation problem of determining $p(y|x)$. We will consider observed data $(x,y)_i,\ i = 1,..,n$, in the infinite data limit $n\rightarrow \infty$, as is standard in identifiability analysis in causal inference, but will \textit{also consider the stability at this ideal point}. This can also be considered as analysing the \textit{approach to} the limit, in addition to the behaviour \textit{at} the limit, but we again emphasise that this is essentially an intrinsic limiting property related to estimability \textit{in principle}.

First, and following \citet{vapnik2013nature} as noted above, we consider the definition of the conditional density $p(y|x)$ as the solution $f(x,y)$ to:

\begin{equation}
    \int_{-\infty}^{y}\int_{-\infty}^{x}f(x,y)dF(x')dy = F(x,y),
\end{equation}
where $F(x)$ and $F(x,y)$ are the cumulative distribution functions of $x$ and $x,y$ respectively. If the relevant densities exist we can write (informally) $dF(x) = p(x)dx$. 

As noted by \citet{vapnik2013nature,Vapnik2015-pz}, as well as e.g. \citet{Donoho1988-zf,Davies2014-dz,Lewbel2016-sk}, the cumulative distributions and the empirical data are naturally directly related in the sense that data generated by two `close' $F$ will also be close, and that the empirical $F$ converges consistently and rapidly to the `true' $F$ (e.g. via the famous Glivenko-Cantelli theorem and related bounds on convergence rates). See for example \citet{Van_der_vaart2013-lx}. Furthermore, there is a 1-1 correspondence between $F$ and the associated probability measure. Hence it is natural to take $F$ as the given and $p(y|x)$ as determined from this via the above integral equation. If a unique solution exists, then $p(y|x)$ is uniquely determined.

The above can be written as a Fredholm integral equation (i.e. with unrestricted integration limits) via:

\begin{equation}
    \int \int H(y-y')H(x-x')f(y',x')dF(x')dy' = F(x,y),
\end{equation}
where $H$ is the Heaviside (step) function with $H(z) = 0$ for $z \geq 1$ and $H(z)=0$ otherwise. This means that, abstractly, we have

\begin{equation}
    Kf = F,
\end{equation}
where the operator $K$ has the kernel $k$:

\begin{equation}
    k(x,x',y,y') = k(x-x',y-y') = H(y-y')H(x-x').
\end{equation}
Thus to determine $p(y|do(x))$ in the identifiable case, and when $p(y|do(x)) = p(y|x)$, we must solve the statistical problem defined by the above integral equation. 

By making the identification $p(y|do(x)) \leftrightarrow \theta$, and since as mentioned above the cumulative distribution function $F(x,y)$ uniquely determines the probability distribution $\mathbb{P}(X,Y)$, we can consider $Kf = F$ to represent our forward mapping, i.e. we make the correspondence 

\begin{equation}
    Kf = F\ \leftrightarrow \ P(\theta) = \mathbb{P}_{\theta}.
\end{equation}
Thus we see that the problem of determining the basic causal quantity $p(y|do(x))$ can be formulated in terms of solving a (Fredholm) integral equation. 

The above covers the case where the data contains variability/errors in $x$ as well as in $y$ and hence naturally leads to considering perturbations to \textit{both} the operator $K$ and the right-hand side $F$, i.e. to consideration of solutions to the perturbed equation 

\begin{equation}
    K_{\delta}f = F_{\delta}.
\end{equation}
This case is considered explicitly in a statistical context by e.g. \cite{vapnik2013nature,Vapnik2015-pz}, though the idea of replacing the operator itself is also at the heart of classical regularisation methods. For simplicity, however, we will instead consider the `conditional' case of estimating $p(y | x)$ for fixed/known/error-free $x$, i.e. we will essentially restrict attention to equations of the form $Kf = F_{\delta}$, where the perturbations are to the right-hand side only. Explicitly, in this setting we again have an integral equation and $p(y | x)$ is, by definition, the solution to

\begin{equation}
    \int H(y-y')f(y',x)dy' = F_{Y|X}(y|x),
\end{equation}
where $F_{Y|X}(y|x) = \mathbb{P}(Y \leq y,X=x)$ and $x$ is considered known. The problem then is to determine the solution $f(y,x) = p(y|x)$ to this equation, given $F$ and $H$. Again, we must solve a (Fredholm) integral equation. We also have the following normalisation and non-negativity constraints for probability densities:  

\begin{equation}
\int f(y,x)dy = \int p(y|x)dy = 1 \text{ for all } x,\text{ and } f(y,x) = p(y|x) \geq 0 \text{ for all } x,y.
\end{equation}
It is entirely possible for the above equations to possess unique but unstable solutions, as we will demonstrate with a simple concrete example, following \cite{vapnik2013nature,Vapnik2015-pz}. This example is also very similar to the example of numerical differentiation given by \citet{engl1996regularization} and the examples considered by \citet{Horowitz2014-us}.

\begin{example}
\label{ex:F}
Consider the case where $F_{Y|X}(y|x) = y$ over $y \in [0,1]$, i.e. we must solve

\begin{equation}
    \int H(y-y')f(y',x)dy' = y
\end{equation}
subject to the normalisation condition over $y \in [0,1]$. It is easy to verify that the solution is:

\begin{equation}
   f(y,x) = p(y | x) = 1.
\end{equation}
Consider next the solution to the equation with a small perturbation to the right-hand side:

\begin{equation}
    \int H(y-y')f_{\delta}(y',x)dy' = y + \delta \sin (\frac{1}{\delta}y)
\end{equation}
for $\delta = \frac{1}{2n\pi}, n \in \mathbb{N}$. In this case the solution is:
\begin{equation}
   f(y,x) = p(y | x) = 1 + \cos(\frac{1}{\delta}y).
\end{equation}
From this, we see that as $\delta \rightarrow 0$, i.e. $n\rightarrow \infty$, we have that the perturbed right-hand side approaches the unperturbed right-hand side, but \textit{the solution to the perturbed equation does not approach the solution to the unperturbed equation}. This example is illustrated in Figure \ref{fig:F}.
\end{example}

\begin{figure}
\begin{center}
 \begin{tikzpicture}[domain=0:4,scale=0.8]
     \draw[<->] (-1.2,0) -- (5.2,0);
     \draw[<->] (0,-0.2) -- (0,5);
     \draw[ samples=200, ultra thick, dashed] plot(\x,{2+2*cos(2*4*pi*\x/4 r)});
     \draw[ samples=200, ultra thick] plot(\x,2);
      \draw[ultra thick] (-1,0) -- (0,0);
      \draw[ultra thick] (4,0) -- (5,0);
      \draw[dashed] (4,0) -- (4,5);
      \draw (4,0) node[below] {1};
      \draw[dashed] (0,2) -- (5.2,2);
      \draw (0,2) node[left] {1};
      \draw[dashed] (0,4) -- (5.2,4);
      \draw (0,4) node[left] {2};
      \draw (0,5) node[above left] {$f$};
      \draw (5.2,0) node[below right] {$y$};
   \end{tikzpicture}
   \hspace{0.5cm}
    \begin{tikzpicture}[domain=0:4,scale=0.8]
     \draw[<->] (-1.2,0) -- (5.2,0);
     \draw[<->] (0,-0.2) -- (0,5);
     \draw[ samples=200, ultra thick, dashed] plot(\x,{\x+4*1/(2*4*pi)*sin(2*4*pi*\x/4 r)});
     \draw[ samples=200, ultra thick] plot(\x,\x);
      \draw[ultra thick] (-1,0) -- (0,0);
      \draw[ultra thick] (4,4) -- (5,4);
      \draw[dashed] (4,0) -- (4,5);
      \draw[dashed] (0,4) -- (4,4);
      \draw (0,4) node[left] {1};
      \draw (4,0) node[below] {1};
      \draw (0,5) node[above left] {$F$};
      \draw (5.2,0) node[below right] {$y$};
   \end{tikzpicture}
\end{center}
\caption{Illustration of Example \ref{ex:F}, i.e. solving an ill-posed integral equation $Kf = F$. Small perturbations to the right-hand side $F$ can give large changes to the solution $f$.}
\label{fig:F}
\end{figure}

The above example demonstrates the ill-posedness of the problem of estimating $p(y|x)$, and hence $p(y | do(x))$ in this case. Of course, ill-posedness can be addressed via regularisation methods, but this amounts to requiring additional restrictions on the model space in which solutions are sought, i.e. restrictions on the causal questions and answers. This is true even \textit{despite identifiability}, i.e. the solution can be unique but arbitrarily unstable.

\subsubsection{Regression}
Regression models can also be written directly as solutions to Fredholm integral equations in a similar manner to the above, or obtained by first estimating the conditional distribution function \citep{vapnik2013nature,Vapnik2015-pz}. Here we simply directly consider the regression function, for fixed $x$, as a functional of the conditional cumulative distribution function. This makes the arguments of \citet{Tibshirani1988-jr,Huber2011robust,Hampel2011-cb,Hampel1971-jx} concerning sensitive functionals and robust statistics directly applicable. 

We again emphasise that, in general, regression functions are not equal to what is sometimes called \textit{causal regression functions} \citep{wasserman2013all}. This latter function captures the `response' $Y$ to `treatment' $X=x$ in the continuous setting. When $X$ is randomly assigned or when e.g. the DAG in Figure \ref{fig:dag} holds, the statistical and causal regression functions are numerically equal, however. For simplicity, we will again assume this relationship holds. Hence we have identifiability, but need to further consider stability in order to assess estimability.

The regression function is defined, for each $x$, by

\begin{equation}
    r(x) = \int y p(y|x)dy = \int y dF_{Y|X}(y|x),
\end{equation}
which defines $r(x)$ as a linear functional of the conditional distribution function of the form $T(F_{Y|X}) = \int \psi dF_{Y|X}$, i.e. we have

\begin{equation}
    r(x) = T(F_{Y|X=x}).
\end{equation}
It can be shown that linear functionals of the form $T(F) = \int \psi dF$ are bounded \textit{iff} $\psi$ is bounded and continuous \citep{Huber2011robust}. 

\begin{example}
\label{ex:mean}
In our case we have $\psi = y$, which is clearly unbounded, and hence $T$ is unbounded (discontinuous). More concretely, it is easy to show that, given any $\epsilon > 0$, the mean of the $\epsilon$-contaminated distribution $Q = (1-\epsilon)F + \epsilon G$ can be made arbitrarily different to that of $F$ by appropriate choice of $G$, despite these distributions being arbitrarily close in terms of contamination distance \citep{Tibshirani1988-jr,Huber2011robust,Hampel2011-cb}. By taking $G$ to be the delta measure at $y$, $\Delta_y$, we obtain the influence function $\text{IF}(y;T,F)$, defined above. It is straightforward to show that for $T$ representing the mean, $\text{IF}(y;T,F) = y$. Hence the gross-error sensitivity, which we recall from above is the supremum of the influence function and is a measure of asymptotic bias, is equal to $+\infty$, i.e. the mean is an arbitrarily sensitive parameter in the nonparametric setting. Thus the mean is essentially \textit{inestimable} in the general setting. 
\end{example}
As discussed above, similar results include those of \citet{Bahadur1956-ms,Donoho1988-zf}. In the causal estimation context this means that any causal quantity corresponding to the mean is a \textit{sensitive causal query} in the truly nonparametric setting. Such queries can be said to be \textit{identifiable but inestimable}.

\subsubsection{Other examples: continuous confounding, average treatment effects, ill-conditioning}
The purpose of the present work is to point out basic conceptual issues in causal inference as seen from an inverse problems perspective. The above examples are intended as elementary and illustrative. Much of (structural/graphical) causal inference work to date appears to be split into conceptual abstract work on topics like identifiability, and concrete implementations in terms of restricted model classes e.g. linear regression. We note, however, that there is increasing interest in truly nonparametric causal inference. One example that was recently brought to our attention is the work by \citet{Miao2018-wi} which discusses the case of a continuous confounder. In their article, the need to solve Fredholm integral equations appears in exactly the same way as we have considered here. They analyse and solve this by the usual tools of inverse problems theory, e.g. by using the singular value decomposition. We emphasise that this leads to additional regularity conditions of exactly the sort discussed in the present article.

 A natural extension to Example \ref{ex:mean} has been considered in the econometrics literature: estimating the average treatment effect (ATE) using propensity scores \citep[see e.g.][for early work]{Rosenbaum1983-ck}. As discussed by \citet{Khan2010-bh}, this leads to a functional of the form\footnote{This functional is potentially nonlinear, but it can be considered as a linear functional of the distribution when the propensity score function is taken as a `given', with values in $(0,1)$, but otherwise arbitrary. In this case the result from \citet{Huber2011robust} -- that linear functionals are bounded \textit{iff} $\psi$ is bounded and continuous -- applies directly.} $\int \psi dF$ where $\psi$ is unbounded unless the propensity score is \textit{bounded away from zero and one}, i.e. does not just lie in $(0,1)$ but instead lies in $(b_l,b_u)$ for some $b_l > 0, b_u <1$. They specifically note that this leads to a subtle interplay between identification and estimability `that limits the practical usefulness of estimators based on these models'. As mentioned previously, \citet{Lewbel2016-sk} provides a comprehensive overview of identification concepts in econometrics and the (often subtle) link between identifiability and estimation. This particularly true of cases \citet{Lewbel2016-sk} labels `identification concepts that affect inference'. \citet{Horowitz2014-us} considers econometrics estimation problems from the perspective of ill-posed inverse problems and gives a number of examples similar to those considered here.

From a slightly different direction \citet{schulman2016stability} consider the `sensitivity of causal identification to small perturbations in the input'. They calculate a condition number for the identification algorithm for identifiable semi-Markovian models, and find that this can be extremely large, i.e. the identification algorithm can be extremely numerically unstable. This illustrates that the issues we have raised here are not restricted to the continuous ill-posed case, but also appear in the form of ill-conditioning in discrete problems. Though we have not done so yet, it would certainly be of interest to consider whether the tools of inverse problems theory can aid in determining which models lead to bad condition numbers and, potentially, how one might introduce appropriate regularisation in such cases to stabilise the algorithm. This would, of course, involve replacing the ill-conditioned problem by a `similar' well-conditioned problem and hence modifying the causal question addressed.

\section{Discussion}
We have considered the concept of estimability and its relation to the concept of identifiability. The (structural/graphical) causal inference literature, in contrast to the inverse problems, statistical, and econometrics literature, takes identifiability as synonymous with estimability. In our view this is a mistake, despite its intuitive appeal. This is supported by the numerous examples of ill-posed problems in the broader literature on inverse problems, statistics, statistical machine learning, and econometrics, of which we have presented only a small number in the present work. In addition to simple examples, we have noted that these issues appear to arise in real causal inference problems \citep{schulman2016stability,Miao2018-wi,Khan2010-bh,Lewbel2016-sk,Escanciano2018-xo}. 

The source of the disconnect between identifiability and estimability (even in principle) is a failure to consider stability to perturbations, even infinitesimally-small ones. In terms of the \citet{hadamard1902problemes} conditions for well-posed problems, identifiability represents uniqueness, but ignores stability. We propose that \textit{true estimability requires additional stability considerations}, i.e. corresponds to considering whether the problem is well-posed. 

These mathematical considerations also raise interesting philosophical questions concerning what types of causal questions are truly answerable in the general setting. For example, \citet{Donoho1988-zf} notes that the types of quantities that are estimable in the truly nonparametric setting concern lower bounds on measures of `model complexity'. Similarly, \citet{vapnik2013nature} relates his learning theory, regularisation via restriction to compact subsets, and restrictions on model class `capacity' (informally a measure of model class complexity) to Popper's concept falsifiability \citep{popper2005logic}. In short, one might say that we can potentially \textit{rule out simple causal explanations, but cannot rule out more complex causal explanations}. Regardless of this interpretation, however, thinking about stability and estimability, even in principle, requires us to go beyond mere identifiability.


\bibliography{identifiability-estimability}

\section*{Acknowledgements}
The writing of this article was stimulated, in large part, by a Twitter discussion the first author participated in concerning identifiability and estimation. The authors would like to thank the other participants: Judea Pearl, Elias Bareinboim, Carlos Cinelli, Alex Breskin, Manjari Narayan, Alexandre Patriota, Karl Rohe, Henning Strandin, Alexander D'Amour, Edward Kennedy, Iv\'{a}n D\'{i}az, and Eric Lofgren (and any others they've forgotten). Alexander D'Amour and Edward Kennedy in particular provided welcome encouragement and discussion of examples, while Iv\'{a}n D\'{i}az mentioned the article by \citet{Miao2018-wi}, and Carlos Cinelli mentioned the article by \citet{schulman2016stability}. Subsequently, Pedro H.C. Sant'Anna provided a number of other useful references from the econometrics literature, a Twitter user named `Sam' pointed to the review article by \citet{Lewbel2016-sk}, and Corey Yanofsky provided helpful feedback. The first author would also like to thank Laurie Davies for many useful discussions about statistics and stability over the last few years, and Sander Greenland for encouraging comments and helpful references.

\end{document}